\documentclass[11pt]{article}

\usepackage{epsfig,epsf,fancybox}
\usepackage{amsmath}
\usepackage{mathrsfs}
\usepackage{amssymb}
\usepackage{graphicx}
\usepackage{color}
\usepackage[linktocpage,pagebackref,colorlinks,linkcolor=blue,anchorcolor=blue,citecolor=blue,urlcolor=blue]{hyperref}
\usepackage{multirow,cite}
\usepackage{booktabs}
\usepackage{algorithm}
\usepackage{algorithmic}
\usepackage[left=1in,top=1in,right=1in,bottom=1in,a4paper]{geometry}
\usepackage{stmaryrd,booktabs}

\numberwithin{equation}{section}

\newtheorem{theorem}{Theorem}[section]

\newtheorem{lemma}[theorem]{Lemma}
\newtheorem{proposition}[theorem]{Proposition}
\newtheorem{definition}[theorem]{Definition}

\newtheorem{assumption}[theorem]{Assumption}
\newtheorem{remark}[theorem]{Remark}

\newcommand{\st}{\textnormal{s.t.}}

\,
\,
\newcommand{\RR}{\mathbb{R}}
\newcommand{\argmin}{\mathop{\rm argmin}}
\newcommand{\LCal}{\mathcal{L}}

\newcommand{\dist}{\mathop{\rm dist}}

\newcommand{\be}{\begin{equation}}
\newcommand{\ee}{\end{equation}}
\newcommand{\ba}{\begin{array}}
\newcommand{\ea}{\end{array}}
\newcommand{\bpm}{\begin{pmatrix}}
\newcommand{\epm}{\end{pmatrix}}

\newcommand{\bea}{\begin{eqnarray}}
\newcommand{\eea}{\end{eqnarray}}

\newcommand{\diam}{\mathrm{diam}}

\newcommand{\etal}{{et al.\ }}

\newcommand{\BCal}{\mathcal{B}}
\newcommand{\CCal}{\mathcal{C}}
\newcommand{\ECal}{\mathcal{E}}

\newcommand{\NCal}{\mathcal{N}}

\newcommand{\TCal}{\mathcal{T}}
\newcommand{\XCal}{\mathcal{X}}

\newcommand{\ZCal}{\mathcal{Z}}

\begin{document}
\title{Structured Nonconvex and Nonsmooth Optimization: \\
Algorithms and Iteration Complexity Analysis}

\author{
Bo Jiang
\thanks{Research Center for Management Science and Data Analytics, School of Information Management and Engineering, Shanghai University of Finance and Economics, Shanghai 200433, China. Research of this author was supported in part by National Natural Science Foundation of China (Grant 11401364).}
\and Tianyi Lin
\thanks{Department of Industrial Engineering and Operations Research,
UC Berkeley, Berkeley, CA 94720, USA.}
\and Shiqian Ma
\thanks{Department of Mathematics, UC Davis, Davis, CA 95616, USA. Research of this author was supported in part by a startup package in Department of Mathematics at UC Davis.}
\and Shuzhong Zhang
\thanks{Department of Industrial and Systems Engineering, University of Minnesota, Minneapolis, MN 55455, USA. Research of this author was supported in part by the National Science Foundation (Grant CMMI-1462408).}
}

\date{\today}

\maketitle

\begin{abstract}
Nonconvex and nonsmooth optimization problems are frequently encountered in much of statistics, business, science and engineering, but they are not yet widely recognized as a {\it technology} in the sense of scalability. A reason for this relatively low degree of popularity is the lack of a well developed system of theory and algorithms to support the applications, as is the case for its convex counterpart. This paper aims to take one step in the direction of {\it disciplined nonconvex and nonsmooth optimization}. In particular, we consider in this paper some constrained nonconvex optimization models in block decision variables, with or without coupled affine constraints. In the absence of coupled constraints, we show a sublinear rate of convergence to an $\epsilon$-stationary solution in the form of variational inequality for a generalized conditional gradient method, where the convergence rate is dependent on the H\"olderian continuity of the gradient of the smooth part of the objective. For the model with coupled affine constraints, we introduce corresponding $\epsilon$-stationarity conditions, and apply two proximal-type variants of the ADMM to solve such a model, assuming the proximal ADMM updates can be implemented for all the block variables except for the last block, for which either a gradient step or a majorization-minimization step is implemented. We show an iteration complexity bound of $O(1/\epsilon^2)$ to reach an $\epsilon$-stationary solution for both algorithms. Moreover, we show that the same iteration complexity of a proximal BCD method follows immediately. Numerical results are provided to illustrate the efficacy of the proposed algorithms for tensor robust PCA.

\vspace{1cm}


\vspace{0.50cm}


\end{abstract}

\newpage

\section{Introduction}\label{sec:intro}
In this paper, we consider the following nonconvex and nonsmooth optimization problem with multiple block variables:
\be\label{noncvx-block-opt} \ba{ll}
\min &  f(x_1,x_2,\cdots, x_N) + \sum\limits_{i=1}^{N-1} r_i(x_i) \\
\st  & \sum_{i=1}^N A_i x_i = b, \ x_i\in \XCal_i, \ i=1,\ldots,N-1,\ea
\ee
where $f$ is differentiable and possibly nonconvex, and each $r_i$ is possibly nonsmooth and nonconvex, $i=1,\ldots, N-1$; $A_i\in\RR^{m\times n_i}$, $b\in\RR^m$, $x_i\in\RR^{n_i}$; and $\XCal_i\subseteq\RR^{n_i}$ are convex sets, $i=1,2,\ldots,N-1$. {One restriction of model~\eqref{noncvx-block-opt} is that the objective function is required to be smooth with respect to the last block variable $x_N$. However,
in Section~\ref{sec:extension} we shall extend the result to cover the general case where $r_N(x_N)$ may be present and that $x_N$ maybe constrained as well.}
A special case of \eqref{noncvx-block-opt} is when the affine constraints are absent, and there is no block structure of the variables (i.e., $x = x_1$ and other block variables do not show up in \eqref{noncvx-block-opt}), which leads to the following more compact form
\be\label{prob:nonconvex}
\min \ \Phi(x) := f(x) + r(x), \ \st \ x \in S \subset\RR^n,
\ee
where $S$ is a convex and compact set.
In this paper, we propose several first-order algorithms for computing an $\epsilon$-stationary solution (to be defined later) for \eqref{noncvx-block-opt} and \eqref{prob:nonconvex}, and analyze their iteration complexities.
{Throughout, we assume the following condition.
\begin{assumption}
The sets of the stationary solutions for \eqref{noncvx-block-opt} and \eqref{prob:nonconvex} are non-empty.
\end{assumption}
}

Problem \eqref{noncvx-block-opt} arises from a variety of interesting applications. For example, one of the nonconvex models for matrix robust PCA can be cast as follows (see, e.g., \cite{Shen-oms-2014}), which seeks to decompose a given matrix $M\in\RR^{m\times n}$ into a superposition of a low-rank matrix $Z$, a sparse matrix $E$ and a noise matrix $B$:
\be\label{M-RPCA} \min_{X,Y,Z,E,B} \ \|Z-XY^\top\|_F^2 + \alpha \mathcal{R}(E), \ \st \ M = Z + E + B, \ \|B\|_F \leq \eta,\ee
where $X\in\RR^{m\times r}$, $Y\in\RR^{n\times r}$, with $r < \min(m,n)$ being the estimated rank of $Z$; $\eta>0$ is the noise level, $\alpha>0$ is a weighting parameter; $\mathcal{R}(E)$
is a regularization function that can improve the sparsity of $E$. One of the widely used regularization functions is the $\ell_1$ norm, which is convex and nonsmooth. However, there are also many nonconvex regularization functions that are widely used in statistical learning and information theory, such as smoothly clipped absolute deviation (SCAD) \cite{Fan-Variable-2001}, log-sum penalty (LSP) \cite{Candes-Enhancing-2008}, minimax concave penalty (MCP) \cite{Zhang-Nearly-2010}, and capped-$\ell_1$ penalty \cite{Zhang-Analysis-2010,Zhang-Multi-2013},
and they are nonsmooth at point $0$ if composed with the absolute value function, which is usually the case in statistical learning. 
Clearly \eqref{M-RPCA} is in the form of \eqref{noncvx-block-opt}. Another example of the form \eqref{noncvx-block-opt} is the following nonconvex tensor robust PCA model (see, e.g., \cite{Xu-APGM-Tucker-2015}), which seeks to decompose a given tensor $\TCal\in\RR^{n_1\times n_2\times\cdots\times n_d}$ into a superposition of a low-rank tensor $\ZCal$, a sparse tensor $\ECal$ and a noise tensor $\BCal$:
\[
\min_{X_i,\CCal,\ZCal,\ECal,\BCal} \ \|\ZCal-\CCal\times_1 X_1\times_2 X_2 \times_3 \cdots \times_d X_d\|_F^2 + \alpha \mathcal{R}( \ECal ), \ \st \ \TCal = \ZCal + \ECal + \BCal, \ \|\BCal\|_F \leq \eta,\]
where $\CCal$ is the core tensor that has a smaller size than $\ZCal$, and $X_i$ are matrices with appropriate sizes, $i=1,\ldots,d$. In fact, the ``low-rank'' tensor in the above model corresponds to the tensor with a small core; however a recent work \cite{Jiang-Yang-Zhang-2016} demonstrates that the CP-rank of the core regardless of its size could be as large as the original tensor. Therefore, if one wants to find the low CP-rank decomposition, then the following model is preferred:
\[
\min_{X_i,\ZCal, \ECal, \BCal} \|\ZCal -\llbracket X_1,X_2,\cdots,X_{d} \rrbracket \|^{2}+\alpha \, \mathcal{R}(\ECal) + \alpha_{\NCal}\| \BCal \|_F^2, \ \st \ \TCal = \ZCal + \ECal + \BCal,
\]
for $X_i = [a^{i,1},a^{i,2},\cdots, a^{i,R}]\in \RR^{n_i \times R}$, $1\leq i \leq d$ and
$\llbracket X_1,X_2,\cdots, X_d\rrbracket : = \sum \limits_{r=1}^{R}a^{1,r}\otimes a^{2,r}\otimes \cdots \otimes a^{d,r},$
where``$\otimes$'' denotes the outer product of vectors, and $R$ is an estimation of the CP-rank.
In addition, the so-called sparse tensor PCA problem \cite{Allen-TPCA-ATSTATS-2012}, which seeks the best sparse rank-one approximation for a given $d$-th order tensor $\TCal$, can also be formulated in the form of \eqref{noncvx-block-opt}:
\[
\min \ -\TCal  (x_1, x_2 , \cdots , x_d) + \alpha \sum_{i=1}^d \mathcal{R}( x_i ), \ \st \ x_i \in S_i = \{x \,|\, \|x\|_2^2 \le 1\}, \  i=1,2,\ldots,d,
\]
where $\TCal  (x_1, x_2 , \cdots , x_d) = \sum_{i_1,\dots, i_d}\TCal _{i_1,\dots, i_d}(x_1)_{i_1}\cdots (x_d)_{i_d}$.

The convergence and iteration complexity for various nonconvex and nonsmooth optimization problems
have recently attracted considerable research attention; see e.g. \cite{Bian-Chen-2013, Bian-Chen-2014, Bian-Chen-Ye-2015, Bolte-characterizations-2010,Attouch-Bolte-Svaiter-2010, Bolte-Sabach-Teboulle-2014, Chen-Ge-Wang-Ye-2014, Curtis-Robinson-Samadi-2016, GeHeHe14, glz13, Liu-Ma-Dai-Zhang-2015, Martinez-Raydan-2016}. In this paper, we study several solution methods that use only the first-order information of the objective function, including a generalized conditional gradient method, variants of alternating direction method of multipliers, and a proximal block coordinate descent method, for solving \eqref{noncvx-block-opt} and \eqref{prob:nonconvex}. Specifically, we apply a generalized conditional gradient (GCG) method to solve \eqref{prob:nonconvex}. We prove that the GCG can find an $\epsilon$-stationary solution for \eqref{prob:nonconvex} in $O(\epsilon^{-q})$ iterations under certain mild conditions, where $q$ is a parameter in the H\"{o}lder condition that characterizes the degree of smoothness for $f$. {In other words, the convergence rate of the algorithm depends on the degree of ``smoothness'' of the objective function.
It should be noted that a similar iteration bound that depends on the parameter $q$ was reported for convex problems \cite{Bredies09}, and for general nonconvex problem, \cite{Bredies09b} analyzed the convergence results, but there was no iteration complexity result.}
Furthermore, we show that if $f$ is concave, then GCG finds an $\epsilon$-stationary solution for \eqref{prob:nonconvex} in $O(1/\epsilon)$ iterations.
For the affinely constrained problem \eqref{noncvx-block-opt},
we propose two algorithms (called proximal ADMM-g and proximal ADMM-m in this paper), both can be viewed as variants of the alternating direction method of multipliers (ADMM). Recently,
there has been an emerging research interest on the ADMM for nonconvex problems (see, e.g., \cite{HongLuoRazaviyayn15,LiPong15,WangYinZeng15,WangCaoXu15,Hong16,Ames-Hong-2016,Yang-Pong-Chen-2016}). However, the results in \cite{LiPong15,WangYinZeng15,WangCaoXu15,Yang-Pong-Chen-2016} only show that the iterates produced by the ADMM converge to a stationary solution without providing an iteration complexity analysis. Moreover, the objective function is required to satisfy the so-called Kurdyka-{\L}ojasiewicz (KL) property~\cite{Kurdyka-1998, Lojasiewicz-1963, Bolte-Daniilidis-Lewis-2006, Bolte-Daniilidis-Lewis-Shiota-2007} to enable those convergence results. In \cite{HongLuoRazaviyayn15}, Hong, Luo and Razaviyayn analyzed the convergence of the ADMM for solving nonconvex consensus and sharing problems. Note that they also analyzed the iteration complexity of the ADMM for the consensus problem. However, they require the nonconvex part of the objective function to be smooth, and nonsmooth part to be convex. In contrast, $r_i$ in our model \eqref{noncvx-block-opt} can be nonconvex and nonsmooth at the same time. Moreover, we allow general constraints $x_i\in\XCal_i,i=1,\ldots,N-1$, while the consensus problem in \cite{HongLuoRazaviyayn15} only allows such constraint for one block variable. A very recent work of Hong \cite{Hong16} discussed the iteration complexity of an augmented Lagrangian method for finding an $\epsilon$-stationary solution for the following problem:
\be\label{prob:nonconvex-linear-constraint}
\min \ f(x), \ \st \ Ax= b, x \in \RR^n,
\ee
under the assumption that $f$ is differentiable.
We will compare our results with \cite{Hong16} in more details in Section \ref{sec:admm}.

Before proceeding, let us first summarize:


{\bf Our contributions.} 

\begin{itemize}
\item[(i)] {We provide definitions of $\epsilon$-stationary solution for \eqref{noncvx-block-opt} and \eqref{prob:nonconvex} using the variational inequalities.} For \eqref{noncvx-block-opt}, our definition of the $\epsilon$-stationary solution
allows each $r_i$ to be nonsmooth and nonconvex.
\item[(ii)] We study a generalized conditional gradient method with a suitable line search rule for solving \eqref{prob:nonconvex}. We assume that the gradient of $f$ satisfies a H\"{o}lder condition, and analyze its iteration complexity for obtaining an $\epsilon$-stationary solution for \eqref{prob:nonconvex}. After we released the first version of this paper, we noticed there are several recent works that study the iteration complexity of conditional gradient method for nonconvex problems. However, our results are different from these. For example, the convergence rate given in \cite{Yu-Zhang-Schuurmans-2014} is worse than ours, and \cite{Lacoste-Julien-2016, Lafond-Wai-Moulines-2016} only consider smooth nonconvex problem with Lipschitz continuous gradient, but our results cover 
    nonsmooth models.
\item[(iii)] We study two ADMM variants (proximal ADMM-g and proximal ADMM-m) for solving \eqref{noncvx-block-opt}, and analyze their iteration complexities for obtaining an $\epsilon$-stationary solution for nonconvex problem \eqref{noncvx-block-opt}. In addition, the setup and the assumptions of our model are different from other recent works. For instance,
\cite{LiPong15} considers a two-block nonconvex problem with an identity coefficient matrix for one block variable in the linear constraint, and requires the coerciveness of the objective or the boundedness of the domain. \cite{WangYinZeng15} assumes that the objective function is coercive over the feasible set
and the nonsmooth objective is {\it restricted prox-regular} or {\it piece-wise linear}. While our algorithm assumes the gradient of the smooth part of the objective function is Lipschitz continuous and the nonsmooth part does not involve the last block variable, which is weaker than the assumptions on the objective functions in \cite{LiPong15, WangYinZeng15}.

\item[(iv)] As an extension, we also show how to use proximal ADMM-g and proximal ADMM-m to find an $\epsilon$-stationary solution for \eqref{noncvx-block-opt} without assuming any condition  on $A_N$.
\item[(v)] When the affine constraints are absent in model \eqref{noncvx-block-opt}, as a by-product,
we demonstrate that the iteration complexity of proximal block coordinate descent (BCD) method with cyclic order can be obtained directly from that of proximal ADMM-g and proximal ADMM-m.
{Although \cite{Bolte-Sabach-Teboulle-2014} gives an iteration complexity result of nonconvex BCD, it requires the KL property, and the complexity depends on a parameter in the KL condition, which is typically unknown.}
\end{itemize}


{\bf Notation.} $\|x\|_2$ denotes the Euclidean norm of vector $x$, and $\|x\|_{H}^2$ denotes $ x^{\top}H x$ for some positive definite matrix $H$. For set $S$ and scalar $p>1$, we denote
$\diam_p(S) := \max_{x,y\, \in S}\|x -y \|_p$,
where $\|x \|_p = (\sum_{i=1}^{n}|x_i|^p)^{1/p}$. Without specification,
we denote $\|x\|=\|x\|_2$ and $\diam(S)=\diam_2(S)$ for short. We use $\dist(x,S)$ to denote the Euclidean distance of vector $x$ to set $S$. Given a matrix $A$, its spectral norm and smallest singular value are denoted by $\|A\|_2$ and $\sigma_{\min}(A)$ respectively.
We use $\lceil a\rceil$ to denote the ceiling of $a$.


{\bf Organization.} The rest of this paper is organized as follows. In Section \ref{sec:gcg} we introduce the notion of $\epsilon$-stationary solution for \eqref{prob:nonconvex} and apply a generalized conditional gradient method to solve \eqref{prob:nonconvex} and analyze its iteration complexity for obtaining an $\epsilon$-stationary solution for  \eqref{prob:nonconvex}. In Section \ref{sec:admm} we give two definitions of $\epsilon$-stationarity for \eqref{noncvx-block-opt} under different settings and propose two ADMM variants that solve \eqref{noncvx-block-opt} and analyze their iteration complexities to reach an $\epsilon$-stationary solution for \eqref{noncvx-block-opt}. In Section \ref{sec:extension} we provide some extensions of the results in Section \ref{sec:admm}. In particular, we first show how to remove some of the conditions that we assume in Section \ref{sec:admm}, and then we apply a proximal BCD method to solve \eqref{noncvx-block-opt} without affine constraints and provide an iteration complexity analysis.
In Section \ref{numerical}, we present numerical results to illustrate the practical efficiency of the proposed algorithms.

\section{A generalized conditional gradient method} 
\label{sec:gcg}

In this section, we study a GCG method for solving \eqref{prob:nonconvex} and analyze its iteration complexity. The conditional gradient (CG) method, also known as the Frank-Wolfe method, was originally proposed in \cite{FrankWolfe56}, and regained a lot of popularity recently due to its capability in solving large-scale problems (see, \cite{FreundGrigas13,Jaggi-ICML-2013,Mu-Zhang-Wright-Goldfarb-2016,Nemirovski-CG-2015,Bach-CG-2015,Beck-CG-2015,Lan-CG-2014}). However, these works focus on solving convex problems. Bredies \etal \cite{Bredies09b} {proved the convergence of a generalized conditional gradient method for solving nonconvex problems in Hilbert space. In this section, by introducing a suitable line search rule, we provide an iteration complexity analysis for this algorithm.}


We make the following assumption in this section regarding \eqref{prob:nonconvex}.
\begin{assumption}\label{assump-p}
In \eqref{prob:nonconvex}, $r(x)$ is convex and nonsmooth, and the constraint set $S$ is convex and compact. Moreover, $f$ is differentiable and there exist some $p>1$ and $\rho >0$ such that
\begin{equation}\label{ineq:p-power}
f(y) \le f(x) + \nabla f(x)^{\top}(y-x) + \frac{\rho}{2}\|y - x\|^p_p, \quad \forall x,y\in S.
\end{equation}
\end{assumption}

The above inequality \eqref{ineq:p-power} is also known as the H\"older condition and was used in other works on first-order algorithms (e.g.,
\cite{Devo-Fran-Nesterov-2013}).
It can be shown that \eqref{ineq:p-power} holds for a variety of functions. {For instance, \eqref{ineq:p-power} holds for any $p$ when $f$ is concave, and is valid for $p=2$
when $\nabla f$ is Lipschitz continuous.}

\subsection{ An $\epsilon$-stationary solution for problem \eqref{prob:nonconvex}}
For smooth unconstrained problem $ \min_x f(x)$, it is natural to define the $\epsilon$-stationary solution using the criterion
$\|\nabla f(x)\|_2\leq \epsilon.$ Nesterov \cite{NesterovConvexBook2004} and Cartis \etal \cite{CartisGouldToint11} showed that the gradient descent type methods with properly chosen step size need $O({1}/{\epsilon^2})$ iterations to find such a solution. Moreover, Cartis \etal \cite{CartisGouldToint10} constructed an example showing that the $O({1}/{\epsilon^2})$ iteration complexity is tight for the steepest descent type algorithm.
However, the case for the constrained nonsmooth nonconvex optimization is subtler. There exist some works on how to define $\epsilon$-optimality condition for the local minimizers of various constrained nonconvex problems \cite{CartisGouldToint13a, Dutta-Deb-Tulshyan-Arora-2013, glz13, Hong16, Ngai-Luc-Thera-2002}. Cartis \etal \cite{CartisGouldToint13a} proposed an approximate measure for smooth problem with convex set constraint.
\cite{Ngai-Luc-Thera-2002} discussed general nonsmooth nonconvex problem in Banach space by using the tool of limiting Fr{\'e}chet $\epsilon$-subdifferential.
{\cite{Dutta-Deb-Tulshyan-Arora-2013} showed that under certain conditions $\epsilon$-KKT solutions can converge to a stationary solution as $\epsilon \to 0$. Here the $\epsilon$-KKT solution is defined by relaxing the complimentary slackness and equilibrium equations of KKT conditions.}
Ghadimi \etal \cite{glz13} considered the following notion of $\epsilon$-stationary solution for \eqref{prob:nonconvex}: 
\be\label{P-S}
P_{S}(x,\gamma):=\frac{1}{\gamma} (x - x^+), \quad \mbox{where }x^+ = \arg\min_{y \in S}{\nabla f(x)^{\top}y + \frac{1}{\gamma}V(y,x) + r(y)},
\ee
where $\gamma>0$ and $V$ is a prox-function. 
They proposed a projected gradient algorithm to solve \eqref{prob:nonconvex} and proved that it takes no more than $O({1}/{\epsilon^2})$ iterations to find an $x$ satisfying
\begin{equation}\label{measure-lan}
\| P_{S}(x,\gamma) \|_2^2 \le \epsilon.
\end{equation}
Our definition of an $\epsilon$-stationary solution for \eqref{prob:nonconvex} is as follows.
\begin{definition}\label{def:eps-nonconvex}
We call $x$ an $\epsilon$-stationary solution ($\epsilon\geq 0$) for \eqref{prob:nonconvex} if the following holds:
\begin{equation}\label{formu:epsilon-KKT}
\psi_S(x):=\inf_{y\in S} \{ \nabla f(x)^{\top}(y-x) + r(y) - r(x) \} \geq - \epsilon .
\end{equation}
If $\epsilon=0$, then $x$ is called a stationary solution for \eqref{prob:nonconvex}.
\end{definition}

Observe that if $r(\cdot)$ is continuous then any cluster point of $\epsilon$-stationary solutions defined above is a stationary solution for \eqref{prob:nonconvex} as $\epsilon \to 0$. Moreover, the stationarity condition is weaker than the usual KKT optimality condition. To see this, we first rewrite \eqref{prob:nonconvex} as the following equivalent unconstrained problem
$$
\min_x f(x) + r(x) + \iota_S(x)
$$
where $\iota_S(x)$ is the indicator function of $S$. Suppose that $x$ is any local minimizer of this problem and thus also a local minimizer of \eqref{prob:nonconvex}.
Since $f$ is differentiable, $r$ and $\iota_S$ are convex, Fermat's rule~\cite{RockafellarWets98} yields
\begin{equation}\label{Fermat-convex-nonsmooth}
0 \in \partial \left( f(x) + r(x) + \iota_S(x) \right) = \nabla f(x) + \partial r(x) + \partial \iota_S(x),
\end{equation}
which further implies that there exists some $z \in \partial r (x)$ such that
$$
(\nabla f(x) + z)^{\top}(y-x)  \geq 0, \quad \forall y\in S.
$$
Using the convexity of $r(\cdot)$, it is equivalent to
\begin{equation}\label{formu:VI-type-KKT}
\nabla f(x)^{\top}(y-x) + r(y) - r(x) \geq 0,\; \forall y\in S.
\end{equation}
Therefore, \eqref{formu:VI-type-KKT} is a necessary condition for local minimum of \eqref{prob:nonconvex} as well.
Furthermore, we claim that $\psi_S(x) \ge -\epsilon$ implies  $\|P_{S}(x,\gamma)\|_2^2 \leq\epsilon/\gamma$ with the prox-function $V(y,x)=\|y-x\|_2^2/2$.
In fact, \eqref{P-S} guarantees that
\begin{equation}\label{formu:general-projection}
\left(\nabla f(x)+\frac{1}{\gamma}(x^+ - x) +z \right)^{\top}(y - x^+) \ge 0, \quad\forall\; y \in S,
\end{equation}
for some $z \in \partial r(x^+)$. By choosing $y = x$ in \eqref{formu:general-projection} one obtains
\be\label{formu:general-projection-1}
\nabla f(x)^{\top}(x - x^+) + r(x) - r(x^+) \ge \left(\nabla f(x) +z \right)^{\top}(x - x^+) \ge \frac{1}{\gamma}\|x^+ - x\|_2^2.
\ee
Therefore, if $\psi_S(x) \geq -\epsilon$, then $\|P_{S}(x,\gamma) \|_2^2 \le \frac{\epsilon}{\gamma} $ holds.

\subsection{The algorithm}\label{Sec:de-alg}

For given point $z$, we define an approximation of the objective function of \eqref{prob:nonconvex} to be:
\be\label{func:linear}
\ell(y;x) :=  f(x) + \nabla f(x)^{\top}(y-x) + r(y),
\ee
which is obtained by linearizing the smooth part (function $f$) of $\Phi$ in \eqref{prob:nonconvex}.
Our GCG method for solving \eqref{prob:nonconvex} is described in Algorithm \ref{alg:gcg}, where $\rho$ and $p$ are from Assumption \ref{assump-p}.
\begin{algorithm}[ht]
\caption{Generalized Conditional Gradient Algorithm (GCG) for solving \eqref{prob:nonconvex}}
\label{alg:gcg}
\begin{algorithmic}[10]
\REQUIRE {Given $x^0 \in S$}
\FOR {$k=0,1,\ldots$}
    \STATE[Step 1] $y^{k}=\arg\min_{y \in S} \ell(y;x^{k})$, and let $d^k = y^k - x^k$;
    \STATE[Step 2] $\alpha_k = \arg\min_{\alpha \in [0,1]}  \alpha\, \nabla f(x^{k})^{\top}d^k + \alpha^p\,\frac{\rho}{2}\|d^k\|^p_p + (1 - \alpha)r(x^k) + \alpha r(y^k)$;
    \STATE[Step 3] Set $x^{k+1}=(1-\alpha_k)x^{k}+ \alpha_k y^{k}$.
\ENDFOR
\end{algorithmic}
\end{algorithm}

{ In each iteration of Algorithm \ref{alg:gcg}, we first perform an exact minimization on the approximated objective function $\ell(y;x)$ to form a direction $d_k$. Then the step size $\alpha_k$ is obtained by an exact line search (which differentiates the GCG from a normal CG method) along the direction $d_k$, where $f$ is approximated by $p$-powered function and the nonsmooth part is replaced by its upper bound. Finally, the iterate is updated by moving along the direction $d_k$ with step size $\alpha_k$.}

{
Note that here we assumed that solving the subproblem in Step 1 of Algorithm \ref{alg:gcg} is relatively easy. That is, we assumed the following assumption. 
\begin{assumption}\label{assump:solve-nonconvex-gcg}
	All subproblems in Step 1 of Algorithm \ref{alg:gcg} can be solved relatively easily.
\end{assumption}

\begin{remark}
Assumption \ref{assump:solve-nonconvex-gcg} is quite common in conditional gradient method. For a list of functions $r$ and sets $S$ such that Assumption \ref{assump:solve-nonconvex-gcg} is satisfied, see \cite{Jaggi-ICML-2013}.
\end{remark}
}

\begin{remark}
It is easy to see that the sequence $\{ \Phi(x^k) \}$ generated by GCG is monotonically nonincreasing~\cite{Bredies09b}, which implies that any cluster point of $\{ x^k \}$ cannot be a strict local maximizer.
\end{remark}

\subsection{An iteration complexity analysis}\label{Sec:iter-analy}
Before we proceed to the main result on iteration complexity of GCG, we need the following lemma that gives a sufficient condition for an $\epsilon$-stationary solution for \eqref{prob:nonconvex}. This lemma is inspired by \cite{GeHeHe14}, and it indicates that if the progress gained by minimizing \eqref{func:linear} is small, then $z$ must already be close to a stationary solution for \eqref{prob:nonconvex}.

\begin{lemma}\label{kkt-criteria}
Define
{$z_{\ell} := \argmin_{x\in S} \ \ell(x;z)$.}
The improvement of the linearization at point $z$ is defined as
\[\triangle \ell_z :=  \ell(z;z) - \ell({z}_{\ell};z)=- \nabla f(z)^{\top}({z}_{\ell}-z) + r(z) - r(z_\ell).\]
Given $\epsilon \geq 0$, for any $z \in S$, if $\triangle \ell_z \le \epsilon$, then $z$ is an $\epsilon$-stationary solution for \eqref{prob:nonconvex} as defined in Definition \ref{def:eps-nonconvex}.
\end{lemma}
\begin{proof} {From the definition of ${z}_{\ell}$}, we have
\[
\ell(y;z) - \ell({z}_{\ell};z) = \nabla f(z)^{\top}(y-{z}_{\ell}) + r(y) - r(z_\ell) \geq 0, \forall y \in S,
\]
which implies that
\begin{eqnarray*}
&  & \nabla f(z)^{\top}(y-z) + r(y) - r(z)\\
&= & \nabla f(z)^{\top}(y-{z}_{\ell}) +  r(y) - r(z_\ell) + \nabla f(z)^{\top}({z}_{\ell}-z) + r(z_\ell) - r(z) \\
&\geq & \nabla f(z)^{\top}({z}_{\ell}-z)+ r(z_\ell) - r(z), \forall y \in S.
\end{eqnarray*}
It then follows immediately that if $\triangle \ell_z \leq \epsilon$, then
$\nabla f(z)^{\top}(y-z)+ r(y) - r(z) \geq - \triangle \ell_z \geq - \epsilon.$
\end{proof}


Denoting $\Phi^*$ to be the optimal value of \eqref{prob:nonconvex}, we are now ready to give the main result of the iteration complexity of GCG (Algorithm \ref{alg:gcg}) for obtaining an $\epsilon$-stationary solution for \eqref{prob:nonconvex}.

\begin{theorem}\label{theorm:CG}
For any $\epsilon \in (0, \diam^p_p(S) \rho) $, GCG finds an $\epsilon$-stationary solution for \eqref{prob:nonconvex} within $\left\lceil \frac{2(\Phi(x^{0}) - \Phi^*)(\diam^p_p(S)\rho)^{q-1}}{\epsilon^q} \right\rceil$ iterations, where $\frac{1}{p} + \frac{1}{q} =1$.
\end{theorem}

\begin{proof}
For ease of presentation, we denote $D:= \diam_p(S)$ and $\triangle \ell^k  := \triangle \ell_{x^k}$.
By Assumption \ref{assump-p}, using the fact that $\frac{\epsilon}{D^p\rho } < 1$, and by the definition of $\alpha_k$ in Algorithm \ref{alg:gcg}, we have
\begin{eqnarray} \label{gcg-proof-1}
 &&({\epsilon}/{(D^p\rho)})^{\frac{1}{p-1}}\triangle \ell^k - \frac{1}{2\rho^{1/(p-1)}} ({\epsilon}/{D})^{\frac{p}{p-1}}  \\
 &\leq & -({\epsilon}/{(D^p\rho)})^{\frac{1}{p-1}} ( \nabla f(x^{k})^{\top}(y^{k}-x^{k})+r(y^{k}) - r(x^{k}) ) \nonumber \\
 && - \frac{\rho}{2} ( {\epsilon}/{(D^p\rho)})^{\frac{p}{p-1}} \|y^{k} - x^{k}\|^p_p \nonumber \\
 &\leq & -\alpha_k \left( \nabla f(x^{k})^{\top}(y^{k}-x^{k})+r(y^{k}) - r(x^{k}) \right) - \frac{\rho \alpha_k^p}{2}\|y^{k} - x^{k}\|^p_p \nonumber  \\
 &\leq & -\nabla f(x^{k})^{\top}(x^{k+1}-x^{k}) + r(x^{k})-r(x^{k+1}) - \frac{\rho}{2}\|x^{k+1} - x^{k}\|^p_p \nonumber \\
 &\leq & f(x^{k}) - f(x^{k+1}) + r(x^{k})-r(x^{k+1}) = \Phi(x^{k}) - \Phi(x^{k+1}), \nonumber
\end{eqnarray}
where the third inequality is due to the convexity of function $r$ and the fact that $x^{k+1} - x^{k} = \alpha_k(y^{k}-x^{k})$, and the last inequality is due to \eqref{ineq:p-power}. 
Furthermore, \eqref{gcg-proof-1} immediately yields
\be\label{gcg-thm-proof-1}
\triangle \ell^k \le  ( {\epsilon}/{(D^p\rho)})^{-\frac{1}{p-1}} (\Phi(x^{k}) - \Phi(x^{k+1}) ) + \frac{\epsilon}{2}.
\ee
For any integer $K>0$, summing \eqref{gcg-thm-proof-1} over $k=0,1,\ldots,K-1$, yields
\begin{align*}
K \min_{k\in\{0,1,\ldots,K-1\}}\triangle \ell^k & \leq   \sum_{k=0}^{K-1}\triangle \ell^k  \leq ( {\epsilon}/{(D^p\rho)})^{-\frac{1}{p-1}} \left(\Phi(x^{0}) - \Phi(x^{K})\right)+ \frac{\epsilon}{2}K \\
& \leq  ( {\epsilon}/{(D^p\rho)})^{-\frac{1}{p-1}} (\Phi(x^{0}) - \Phi^* )+ \frac{\epsilon}{2}K,
\end{align*}
where $\Phi^*$ is the optimal value of \eqref{prob:nonconvex}.
It is easy to see that by setting $K = \left\lceil \frac{2(\Phi(x^{0}) - \Phi^*)(D^p\rho)^{q-1}}{\epsilon^q} \right\rceil$, the above inequality implies $\triangle \ell_{x^{k^*}} \leq \epsilon$, where $k^*\in \argmin_{k\in\{0,\ldots,K-1\}}\triangle\ell^k$. According to Lemma \ref{kkt-criteria}, $x^{k^*}$ is an $\epsilon$-stationary solution for \eqref{prob:nonconvex} as defined in Definition \ref{def:eps-nonconvex}.
\end{proof}

Finally, if $f$ is concave, then the iteration complexity can be improved as $O(1/\epsilon)$.
\begin{proposition}\label{coro:concave}
Suppose that $f$ is a concave function. If we set $\alpha_k = 1$ for all $k$ in GCG (Algorithm \ref{alg:gcg}), then it returns an $\epsilon$-stationary solution for \eqref{prob:nonconvex} within $\left\lceil \frac{\Phi(x^{0}) - \Phi^*}{\epsilon} \right\rceil$ iterations.
\end{proposition}

\begin{proof}
By setting $\alpha_k = 1$ in Algorithm \ref{alg:gcg} we have $x^{k+1}=y^k$ for all $k$. Since $f$ is concave, it holds that
\[\triangle \ell^k = -\nabla f(x^{k})^{\top}(x^{k+1}-x^{k})+ r(x^{k})-r(x^{k+1})\leq \Phi(x^{k}) - \Phi(x^{k+1}).\]
Summing this inequality over $k=0,1,\ldots,K-1$ yields
$K \min_{k\in\{0,1,\ldots,K-1\}}\triangle \ell^k  \leq \Phi(x^{0}) - \Phi^*,$
which leads to the desired result immediately.
\end{proof}

\section{Variants of ADMM for solving nonconvex problems with affine constraints}\label{sec:admm}

In this section, we study two variants of the ADMM (Alternating Direction Method of Multipliers) for solving the general problem \eqref{noncvx-block-opt}, and analyze their iteration complexities for obtaining an $\epsilon$-stationary solution (to be defined later) under certain conditions. Throughout this section, the following two assumptions regarding problem \eqref{noncvx-block-opt} are assumed.
\begin{assumption}\label{assump:Lipschitz}
{
The gradient of the function $f$ is Lipschitz continuous with Lipschitz constant $L>0$, i.e., for any $(x_1^1,\cdots,x_N^1)$ and $(x_1^2,\cdots,x_N^2)\in\XCal_1\times\cdots\times\XCal_{N-1}\times\RR^{n_N}$, it holds that
\be\label{f-Lipschitz}
\left\| \nabla f(x_1^1,x_2^1,\cdots, x_N^1) - \nabla f(x_1^2,x_2^2,\cdots, x_N^2) \right\| \leq L\left\| \left( x_1^1 - x_1^2, x_2^1 - x_2^2, \cdots, x_N^1 - x_N^2\right)\right\|,
\ee
}
which implies that for any $(x_1,\cdots,x_{N-1})\in\XCal_1\times\cdots\times\XCal_{N-1}$ and $x_N$, $\hat{x}_N\in\RR^{n_N}$, we have
\be\label{decent-lemma}
f(x_1,\cdots,x_{N-1}, x_N)  \le f(x_1,\cdots,x_{N-1}, \hat{x}_N) +  (x_N - \hat{x}_N)^{\top}\nabla_N f(x_1,\cdots,x_{N-1}, \hat{x}_N) + \frac{L}{2}\|x_N - \hat{x}_N\|^2.
\ee
\end{assumption}

\begin{assumption}\label{value-lower-bound}
$f$ and $r_i, i=1,\ldots,N-1$ are all lower bounded over the appropriate domains defined via the sets $\XCal_1, \XCal_2, \cdots, \XCal_{N-1}, \RR^{n_N}$, and we denote
$$f^* = \inf\limits_{x_i\in\XCal_i,i=1,\ldots,N-1; x_N\in\RR^{n_N}} \ \{f(x_1,x_2,\cdots,x_N)\}$$
and $r_i^* = \inf\limits_{x_i\in\XCal_i} \ \{r_i(x_i)\}$ for $i=1,2,\ldots,N-1$.
\end{assumption}

\subsection{Preliminaries}

To characterize the optimality conditions for \eqref{noncvx-block-opt} when $r_i$ is nonsmooth and nonconvex, we need to recall the notion of the generalized gradient (see, e.g., \cite{RockafellarWets98}).
\begin{definition}\label{Def:general-subgradient}
Let $h:\RR^{n} \to \RR \cup \{+\infty\}$ be a proper lower semi-continuous function. Suppose $h(\bar{x})$ is finite for a given $\bar{x}$. For $v \in \RR^n$, we say that \\
(i). $v$ is a regular subgradient (also called Fr$\acute{e}$chet subdifferential) of $h$ at $\bar{x}$, written $v \in \hat{\partial}h(\bar{x})$, if
\[\lim_{x \neq \bar{x}}\inf_{x \to \bar{x}}\frac{h(x) - h(\bar{x}) - \langle v, x - \bar{x} \rangle}{\|x - \bar{x}\|} \geq 0;\]
(ii). $v$ is a general subgradient of $h$ at $\bar{x}$, written $v \in \partial h(\bar{x})$, if there exist sequences $\{x^k\}$ and $\{v^k\}$ such that $x^{k} \to \bar{x}$ with $h(x^{k}) \to h(\bar{x})$, and $v^k \in \hat{\partial}h(x^k)$ with $v^k \to v$ when $k \to \infty$.
\end{definition}

The following proposition lists some well-known facts about the lower semi-continuous functions.

\begin{proposition}\label{nonsmooth-property} Let $h:\RR^{n} \to \RR \cup \{+\infty\}$ and $g:\RR^{n} \to \RR \cup \{+\infty\}$ be proper lower semi-continuous functions. Then it holds that:\\
(i) (Theorem 10.1 in \cite{RockafellarWets98}) Fermat's rule remains true: if $\bar{x}$ is a local minimum of $h$, then $0 \in \partial h(\bar{x})$.\\
(ii) If $h(\cdot)$ is continuously differentiable at $x$, then $\partial (h + g) (x) = \nabla h(x) + \partial g(x)$.\\
(iii) (Exercise 10.10 in \cite{RockafellarWets98}) If $h$ is locally Lipschitz continuous at ${x}$, then $\partial (h + g) (x) \subset \partial h(x) + \partial g(x)$. \\
(iv) Suppose $h(x)$ is locally Lipschitz continuous, $X$ is a closed and convex set, and
$\bar{x}$ is a local minimum of $h$ on X. Then there exists $v \in \partial h(\bar{x})$ such that
$(x - \bar{x})^{\top} v \geq 0,\forall x\in X$.
\end{proposition}


In our analysis, we frequently use the following identity that holds for any vectors $a,b,c,d$,
\be\label{triangle-identity}
(a-b)^\top(c-d) = \frac{1}{2}\left(\|a-d\|_2^2-\|a-c\|_2^2+\|b-c\|_2^2-\|b-d\|_2^2\right).
\ee

\subsection{ An $\epsilon$-stationary solution for problem \eqref{noncvx-block-opt}}

We now introduce notions of $\epsilon$-stationarity for \eqref{noncvx-block-opt} under the following two settings:
(i) {\bf Setting 1}: $r_i$ is Lipschitz continuous, and $\XCal_i$ is a compact set, for $i=1,\ldots,N-1$; (ii) {\bf Setting 2}: $r_i$ is lower semi-continuous, and $\XCal_i = \RR^{n_i}$, for $i=1,\ldots,N-1$.

\begin{definition}[$\epsilon$-stationary solution for \eqref{noncvx-block-opt} in Setting 1]\label{r-nonconvex-constrained}
Under the conditions in {\bf Setting~1}, for $\epsilon\geq 0$, we call ${\left(x_1^*,\cdots, x_N^* \right)}$ an $\epsilon$-stationary solution for \eqref{noncvx-block-opt} {if there exists a Lagrange multiplier $\lambda^*$ such that
the following holds for any $\left(x_1,\cdots,x_N\right)\in\XCal_1\times\cdots\times\XCal_{N-1}\times\RR^{n_N}$}:
\begin{eqnarray}
\left(x_i-{x}^*_i\right)^\top\left[ {g}^*_i + \nabla_i f(x_1^*,\cdots, x^*_N)  - A_i^\top {\lambda}^*\right]  &\geq& -\epsilon,  \quad i=1,\ldots,N-1, \label{S2-Prob-Opt-1}\\
\left\| \nabla_N f(x_1^*,\ldots, x_{N-1}^*, x_N^*) - A_{N}^\top \lambda^* \right\|& \le & \epsilon,  \label{S1-Prob-Opt-2}\\
\left\| \sum_{i=1}^{N} A_i x_i^* - b \right\|& \le & \epsilon, \label{S1-Prob-Opt-3}
\end{eqnarray}
where ${g}^*_i$ is a general subgradient of $r_i$ at point $x_i^*$.
If $\epsilon=0$, we call ${ \left(x_1^*,\cdots, x_N^* \right) }$ a stationary solution for \eqref{noncvx-block-opt}.
\end{definition}


If $\XCal_i = \RR^{n_i}$ for $i=1,\ldots, N-1$, then the VI style conditions in Definition \ref{r-nonconvex-constrained} reduce to the following.
\begin{definition}[$\epsilon$-stationary solution for \eqref{noncvx-block-opt} in Setting 2]\label{r-nonconvex-unconstrained}
Under the conditions in {\bf Setting~2}, for $\epsilon\geq 0$, we call ${ \left(x_1^*,\ldots, x_N^*\right) }$ to be an $\epsilon$-stationary solution for \eqref{noncvx-block-opt} {if there exists a Lagrange multiplier $\lambda^*$ such that \eqref{S1-Prob-Opt-2}, \eqref{S1-Prob-Opt-3} and the following holds for any $\left(x_1,\cdots,x_N\right)\in\XCal_1\times\cdots\times\XCal_{N-1}\times\RR^{n_N}$}:
\be\label{S3-Prob-Opt-1}
\dist\left(-\nabla_i f(x_1^*,\cdots, x^*_N)  + A_i^\top\lambda^*, {\partial} r_i(x_i^*)\right) \leq \epsilon, \ i=1,\ldots,N-1,
\ee
where ${\partial} r_i(x_i^*)$ is the general subgradient of $r_i$ at $x_i^*$, $i=1,2,\ldots,N-1$. If $\epsilon=0$, we call ${ \left(x_1^*,\cdots, x_N^*\right) }$ to be a stationary solution for \eqref{noncvx-block-opt}.
\end{definition}

The two settings of problem \eqref{noncvx-block-opt} considered in this section and their corresponding definitions of $\epsilon$-stationary solution, are summarized in Table \ref{table:epsilon-stationary}.

\begin{table}[htbp]
\center\vspace{-1em}
\caption{$\epsilon$-stationary solution of \eqref{noncvx-block-opt} in two settings}
\begin{tabular}{|l|c|c|c|c|c|c|c|} \hline
 & $r_i$, $i=1,\ldots, N-1$  & $\XCal_i$, $i=1,\ldots, N-1$ &  $\epsilon$-stationary solution \\ \hline
 {\bf Setting $1$}   &  Lipschitz continuous & $\XCal_i \subset \RR^{n_i}$ compact & Definition~\ref{r-nonconvex-constrained}  \\ \hline
 {\bf Setting $2$} & lower semi-continuous & $\XCal_i = \RR^{n_i}$ & Definition~\ref{r-nonconvex-unconstrained} \\ \hline
\end{tabular}\label{table:epsilon-stationary}
\end{table}


A very recent work of Hong \cite{Hong16} proposes a definition of an $\epsilon$-stationary solution for problem \eqref{prob:nonconvex-linear-constraint}, and analyzes the iteration complexity of a proximal augmented Lagrangian method for obtaining such a solution.
Specifically, $(x^*,\lambda^*)$ is called an $\epsilon$-stationary solution for \eqref{prob:nonconvex-linear-constraint} in \cite{Hong16} if $Q(x^*,\lambda^*) \le \epsilon$, where
\[Q(x,\lambda) := \| \nabla_x\LCal_\beta (x, \lambda) \|^2 + \| Ax - b\|^2,\]
and $\LCal_\beta(x,\lambda) := f(x) - \lambda^\top \left( Ax - b\right) + \frac{\beta}{2}\left\| Ax - b\right\|^2$ is the augmented Lagrangian function of \eqref{prob:nonconvex-linear-constraint}. Note that \cite{Hong16} assumes that $f$ is differentiable and has bounded gradient in \eqref{prob:nonconvex-linear-constraint}.
It is easy to show that an $\epsilon$-stationary solution in \cite{Hong16} is equivalent to an $O(\sqrt{\epsilon})$-stationary solution for \eqref{noncvx-block-opt} according to Definition \ref{r-nonconvex-unconstrained} with $r_i=0$ and $f$ being differentiable. Note that there is no set constraint in \eqref{prob:nonconvex-linear-constraint}, and so the notion of the $\epsilon$-stationarity in \cite{Hong16} is not applicable in the case of Definition \ref{r-nonconvex-constrained}.

\begin{proposition}\label{prop:relation-to-Hong}
Consider the $\epsilon$-stationary solution in Definition \ref{r-nonconvex-unconstrained} applied to problem~\eqref{prob:nonconvex-linear-constraint}, i.e., one block variable and $r_i(x)=0$. Then $x^*$ is a $\gamma_1\sqrt{{\epsilon}}$-stationary solution in Definition \ref{r-nonconvex-unconstrained}, with { Lagrange multiplier $\lambda^*$ and }$\gamma_1 = 1/(\sqrt{2\beta^2\|A\|_2^2+3})$, implies $Q(x^*,\lambda^*) \le \epsilon$. On the contrary, if $Q(x^*,\lambda^*) \le \epsilon$, then $x^*$ is a $\gamma_2\sqrt{{\epsilon}}$-stationary solution from Definition \ref{r-nonconvex-unconstrained} {with Lagrange multiplier $\lambda^*$}, where $\gamma_2 = \sqrt{2 (1 + \beta^2 \|A\|_2^2)}$.
\end{proposition}

\begin{proof}
Suppose $x^*$ is a $\gamma_1\sqrt{{\epsilon}}$-stationary solution as defined in Definition \ref{r-nonconvex-unconstrained}. We have
$\| \nabla f(x^*) - A^{\top}\lambda^* \| \leq \gamma_1\sqrt{\epsilon}$ and $\| Ax^* - b\| \leq \gamma_1\sqrt{\epsilon}$,
which implies that
\begin{eqnarray*}
Q(x^*,\lambda^*) &=& \| \nabla f(x^*) - A^{\top}\lambda^* + \beta A^{\top} (Ax^* - b)\|^2 + \| Ax^* - b\|^2\\
& \le & 2 \| \nabla f(x^*) - A^{\top}\lambda^* \|^2 + 2 \beta^2 \|A\|_2^2\| Ax^* - b\|^2+ \| Ax^* - b\|^2\\
& \le & 2\gamma_1^2 \epsilon + (2 \beta^2 \|A\|_2^2 + 1)\gamma_1^2 \epsilon = \epsilon.
\end{eqnarray*}
On the other hand, if $Q(x^*,\lambda^*)\leq \epsilon$, then we have
$\| \nabla f(x^*) - A^{\top}\lambda^* + \beta A^{\top} (Ax^* - b)\|^2 \le \epsilon$ and $\| Ax^* - b\|^2 \le \epsilon$. Therefore,
\begin{eqnarray*}
\|\nabla f(x^*) - A^{\top}\lambda^*\|^2 &\le& 2\| \nabla f(x^*) - A^{\top}\lambda^* + \beta A^{\top} (Ax^* - b)\|^2 + 2\|-\beta A^{\top} (Ax^* - b)\|^2\\
&\le& 2\| \nabla f(x^*) - A^{\top}\lambda^* + \beta A^{\top} (Ax^* - b)\|^2 + 2\beta^2 \|A\|_2^2 \|Ax^* - b\|^2\\
&\le &2 (1 + \beta^2 \|A\|_2^2)\,\epsilon.
\end{eqnarray*}
The desired result then follows immediately.
\end{proof}

In the following, we introduce two variants of ADMM, to be called proximal ADMM-g and proximal ADMM-m, that solve \eqref{noncvx-block-opt} under some additional assumptions on $A_N$. In particular, proximal ADMM-g assumes $A_N=I$, and proximal ADMM-m assumes $A_N$ to have full row rank.


\subsection{Proximal gradient-based ADMM 
(proximal ADMM-g)}

Our proximal ADMM-g solves \eqref{noncvx-block-opt} under the condition that $A_N=I$. {In this case, the problem reduces to a so-called sharing problem in the literature which has the following form
$$
\ba{ll}
\min & f(x_1,\ldots,x_N) + \sum\limits_{i=1}^{N-1} r_i(x_i) \\
\st  & \sum_{i=1}^{N-1} A_i x_i + x_N = b, \ x_i\in \XCal_i, \ i=1,\ldots,N-1.\ea
$$
For applications of the sharing problem, see \cite{Boyd2011,HongLuoRazaviyayn15,Lin-Ma-Zhang-2015-free-gamma,Lin-Ma-Zhang-2015}.
}
Our proximal ADMM-g for solving \eqref{noncvx-block-opt} with $A_N=I$ is described in Algorithm \ref{alg:gadm}. It can be seen from Algorithm \ref{alg:gadm} that proximal ADMM-g is based on the framework of augmented Lagrangian method, and can be viewed as a variant of the ADMM.
The augmented Lagrangian function of \eqref{noncvx-block-opt} is defined as
\[\LCal_\beta(x_1,\cdots,x_N,\lambda):=f(x_1,\cdots,x_N)+\sum_{i=1}^{N-1}r_i(x_i)-\left\langle\lambda,\sum_{i=1}^N A_ix_i-b \right\rangle+\frac{\beta}{2}\left\|\sum_{i=1}^NA_ix_i-b\right\|_2^2,\]
where $\lambda$ is the Lagrange multiplier associated with the affine constraint, and $\beta>0$ is a penalty parameter.
In each iteration, proximal ADMM-g minimizes the augmented Lagrangian function plus a proximal term for block variables $x_1,\ldots,x_{N-1}$, with other variables being fixed; and then a gradient descent step is conducted for $x_N$, and finally the Lagrange multiplier $\lambda$ is updated. The interested readers are referred to \cite{GaoJiangZhang17} for gradient-based ADMM and its various stochastic variants for convex optimization.


\begin{algorithm}[ht]
\caption{Proximal Gradient-based ADMM (proximal ADMM-g) for solving \eqref{noncvx-block-opt} with $A_N=I$}
\label{alg:gadm}
\begin{algorithmic}[10]
\REQUIRE {Given $\left(x_1^0,x_2^0,\cdots,x_N^0\right)\in\XCal_1\times\cdots\times\XCal_{N-1}\times\RR^{n_N}$, $\lambda^0\in\RR^m$}
\FOR {$k=0,1,\ldots$}
    \STATE[Step 1] $x_i^{k+1} := \argmin_{x_i\in\XCal_i} \ \LCal_\beta(x_1^{k+1},\cdots,x_{i-1}^{k+1},x_i,x_{i+1}^k,\cdots,x_N^k,\lambda^k) + \frac{1}{2}\left\| x_i - x_i^k\right\|^2_{H_i}$ for some positive definite matrix $H_i$, $\ i=1,\ldots,N-1$
    \STATE[Step 2] $x_N^{k+1}  := x_N^k - \gamma\nabla_N \LCal_\beta(x_1^{k+1}, x_2^{k+1}, \cdots, x_N^k, \lambda^k)$
    \STATE[Step 3] $\lambda^{k+1}  := \lambda^k - \beta\left( \sum_{i=1}^N A_i x_i^{k+1} - b\right)$
\ENDFOR
\end{algorithmic}
\end{algorithm}

{

\begin{remark}
Note that here we actually assumed that all subproblems in Step 1 of Algorithm \ref{alg:gadm}, though possibly nonconvex, can be solved to global optimality. Many important problems arising from statistics satisfy this assumption.
In fact, when the coupled objective is absent or can be linearized, after choosing some proper matrix $H_i$, the solution of the corresponding subproblem is given by the proximal mappings of $r_i$. As we mentioned earlier, many nonconvex regularization functions such as SCAD, LSP, MCP and Capped-$\ell_1$ admit closed-form proximal mappings. Moreover, in Algorithm \ref{alg:gadm}, we can choose
\be\label{beta-lower-bound}\beta > \max\left( \frac{18\sqrt{3}+6}{13}L,\; \max\limits_{i=1,2,\ldots,N-1}\frac{6 L^2}{\sigma_{\min}(H_i)}\right),\ee
and
\be\label{lemma:proximal ADMM-g-monotonicity-gamma}\gamma \in \left( \frac{13\beta -\sqrt{13\beta^2 - 12\beta L - 72L^2}}{6L^2 + \beta L + 13\beta^2},  \frac{13\beta + \sqrt{13\beta^2 - 12\beta L - 72L^2}}{6L^2 + \beta L + 13\beta^2} \right)\ee
which guarantee the convergence rate of the algorithm as shown in Lemma~\ref{lemma:proximal ADMM-g-monotonicity-gamma} and Theorem~\ref{Thm:proximal ADMM-g-complexity}.
\end{remark}
}

Before presenting the main result on the iteration complexity of proximal ADMM-g, we need some lemmas.
\begin{lemma}
Suppose the sequence $\{(x_1^k,\cdots,x_N^k , \lambda^k)\}$ is generated by Algorithm \ref{alg:gadm}.
The following inequality holds
\begin{eqnarray}\label{lambda-norm-bound-I}
 \| \lambda^{k+1} - \lambda^k \|^2 &\leq& 3(\beta - 1/\gamma)^2 \| x_N^k - x_N^{k+1} \|^2  \nonumber \\
&& \quad + 3 ((\beta - 1/\gamma)^2 + L^2)\|x_N^{k-1} - x_N^k \|^2 + 3L^2\sum_{i=1}^{N-1} \| x_i^{k+1} - x_i^k\|^2.
\end{eqnarray}
\end{lemma}
\begin{proof}
Note that Steps 2 and 3 of Algorithm \ref{alg:gadm} yield that
\be\label{lambda-xN-I}
\lambda^{k+1} =  (\beta - 1/\gamma) (x_N^k - x_N^{k+1}) +  \nabla_N f(x_1^{k+1},\cdots,x_{N-1}^{k+1},x_N^k).
\ee
Combining \eqref{lambda-xN-I} and \eqref{f-Lipschitz} yields that
\begin{eqnarray*}
& & \| \lambda^{k+1} - \lambda^k \|^2 \\
&\leq &  \|  (\nabla_N f(x_1^{k+1},\cdots,x_{N-1}^{k+1},x_N^k) - \nabla_N f(x_1^k,\cdots,x_{N-1}^k,x_N^{k-1})  ) + (\beta - 1/\gamma) (x_N^k - x_N^{k+1}) \\
& & - (\beta - 1/\gamma)(x_N^{k-1} - x_N^k) \|^2 \\
&\leq & 3 \| \nabla_N f(x_1^{k+1},\cdots,x_{N-1}^{k+1},x_N^k) - \nabla_N f(x_1^k,\cdots,x_{N-1}^k,x_N^{k-1}) \|^2 + 3(\beta - 1/\gamma)^2 \| x_N^k - x_N^{k+1}\|^2 \\
& & + 3\left[ \beta - \frac{1}{\gamma}\right]^2\left\|  x_N^{k-1} - x_N^k \right\|^2 \\
&\leq & 3\left[ \beta - \frac{1}{\gamma}\right]^2\left\| x_N^k - x_N^{k+1} \right\|^2 + 3\left[ \left(\beta - \frac{1}{\gamma}\right)^2 + L^2\right]\left\|  x_N^{k-1} - x_N^k \right\|^2 + 3L^2\sum_{i=1}^{N-1} \left\| x_i^{k+1} - x_i^k\right\|^2.
\end{eqnarray*}
\end{proof}

We now define the following function, which will play a crucial role in our analysis:
\be\label{potential-function}\Psi_G\left(x_1,x_2,\cdots,x_N,\lambda,\bar{x}\right) = \LCal_\beta(x_1,x_2,\cdots,x_N, \lambda) + \frac{3}{\beta}\left[ \left(\beta - \frac{1}{\gamma}\right)^2 + L^2\right] \left\| x_N - \bar{x}\right\|^2.\ee

\begin{lemma}\label{lemma:proximal ADMM-g-monotonicity}
Suppose the sequence $\{(x_1^k,\cdots,x_N^k , \lambda^k)\}$ is generated by Algorithm \ref{alg:gadm},
 where the parameters  $\beta$ and $\gamma$ are taken according to \eqref{beta-lower-bound} and \eqref{lemma:proximal ADMM-g-monotonicity-gamma} respectively.
Then $\Psi_G(x_1^{k+1}, \cdots, x_N^{k+1}, \lambda^{k+1}, x_N^k)$ monotonically decreases over $k\geq 0$.
\end{lemma}

\begin{proof}
From Step 1 of Algorithm \ref{alg:gadm} it is easy to see that
\be\label{lemma:proximal ADMM-g-monotonicity-proof-1}
\LCal_\beta\left(x_1^{k+1},\cdots, x_{N-1}^{k+1}, x_N^k, \lambda^k\right) \leq \LCal_\beta\left(x_1^k, \cdots, x_N^k, \lambda^k\right) - \sum\limits_{i=1}^{N-1}\frac{1}{2}\left\| x_i^k -x_i^{k+1}\right\|^2_{H_i}.
\ee
From Step 2 of Algorithm \ref{alg:gadm} we get that
\begin{equation}\label{lemma:proximal ADMM-g-monotonicity-proof-2}
\ba{lll}
0 & = & \left(x_N^k - x_N^{k+1}\right)^\top\left[ \nabla f(x_1^{k+1}, \cdots, x_{N-1}^{k+1}, x_N^k) - \lambda^k + \beta\left( \sum_{i=1}^{N-1} A_i x_i^{k+1} + x_N^k - b \right) - \frac{1}{\gamma}\left( x_N^k - x_N^{k+1}\right) \right]   \\
& \leq & f(x_1^{k+1},\cdots, x_{N-1}^{k+1}, x_N^k) - f(x_1^{k+1}, \cdots, x_N^{k+1}) + \frac{L}{2}\left\| x_N^k - x_N^{k+1} \right\|^2 - \left( x_N^k - x_N^{k+1}\right)^\top\lambda^k   \\
& & + \frac{\beta}{2}\left\| x_N^k - x_N^{k+1}\right\|^2 + \frac{\beta}{2}\left\| \sum\limits_{i=1}^{N-1} A_i x_i^{k+1} + x_N^k - b\right\|^2 - \frac{\beta}{2}\left\| \sum\limits_{i=1}^{N-1} A_i x_i^{k+1} + x_N^{k+1} - b\right\|^2 - \frac{1}{\gamma}\left\| x_N^k - x_N^{k+1}\right\|^2  \\
& = & \LCal_\beta( x_1^{k+1},\cdots,x_{N-1}^{k+1}, x_N^k, \lambda^k) - \LCal_\beta(x_1^{k+1}, \cdots, x_N^{k+1}, \lambda^k) + \left(\frac{L+\beta}{2} - \frac{1}{\gamma}\right)\left\| x_N^k - x_N^{k+1}\right\|^2,
\ea
\end{equation}
where the inequality follows from \eqref{decent-lemma} and \eqref{triangle-identity}.
Moreover, the following equality holds trivially
\be\label{lemma:proximal ADMM-g-monotonicity-proof-3}
\LCal_\beta(x_1^{k+1}, \cdots, x_N^{k+1}, \lambda^{k+1}) = \LCal_\beta(x_1^{k+1}, \cdots, x_N^{k+1}, \lambda^k) + \frac{1}{\beta}\left\|\lambda^k - \lambda^{k+1}\right\|^2.
\ee
Combining \eqref{lemma:proximal ADMM-g-monotonicity-proof-1}, \eqref{lemma:proximal ADMM-g-monotonicity-proof-2}, \eqref{lemma:proximal ADMM-g-monotonicity-proof-3} and \eqref{lambda-norm-bound-I} yields that
\begin{eqnarray*}
& & \LCal_\beta(x_1^{k+1}, \cdots, x_N^{k+1}, \lambda^{k+1}) - \LCal_\beta(x_1^k, \cdots, x_N^k, \lambda^k) \\
& \leq & \left(\frac{L+\beta}{2} - \frac{1}{\gamma}\right)\left\| x_N^k - x_N^{k+1}\right\|^2 - \sum\limits_{i=1}^{N-1}\frac{1}{2}\left\| x_i^k -x_i^{k+1}\right\|^2_{H_i} + \frac{1}{\beta}\left\|\lambda^k - \lambda^{k+1}\right\|^2 \\
& \leq & \left(\frac{L+\beta}{2} - \frac{1}{\gamma} + \frac{3}{\beta}\left[ \beta - \frac{1}{\gamma}\right]^2\right)\left\| x_N^k - x_N^{k+1}\right\|^2 + \frac{3}{\beta}\left[ \left(\beta - \frac{1}{\gamma}\right)^2 + L^2\right]\left\| x_N^{k-1} - x_N^k \right\|^2 \\
& & + \sum\limits_{i=1}^{N-1} \left(x_i^k - x_i^{k+1}\right)^{\top}\left( \frac{3L^2}{\beta}I - \frac{1}{2}H_i\right)\left(x_i^k - x_i^{k+1}\right),
\end{eqnarray*}
which further implies that
\begin{eqnarray}
& & \Psi_G(x_1^{k+1}, \cdots, x_N^{k+1}, \lambda^{k+1}, x_N^k) - \Psi_G(x_1^k, \cdots, x_N^k, \lambda^k, x_N^{k-1})\label{bound-potential-function-I} \\
& \leq & \left(\frac{L+\beta}{2} - \frac{1}{\gamma} + \frac{6}{\beta}\left[ \beta - \frac{1}{\gamma}\right]^2 + \frac{3L^2}{\beta}\right)\left\| x_N^k - x_N^{k+1}\right\|^2 \nonumber \\
&& \qquad - \sum\limits_{i=1}^{N-1}\left\| x_i^k - x_i^{k+1}\right\|^2_{  \frac{1}{2}H_i - \frac{3L^2}{\beta}I } . \nonumber
\end{eqnarray}
It is easy to verify that when $\beta > \frac{18\sqrt{3}+6}{13}L$, then $\gamma$ defined as in \eqref{lemma:proximal ADMM-g-monotonicity-gamma} ensures that $\gamma>0$ and
\begin{equation}\label{beta-L-decrease}
\frac{L+\beta}{2} - \frac{1}{\gamma} + \frac{6}{\beta}\left[ \beta - \frac{1}{\gamma}\right]^2 + \frac{3L^2}{\beta} < 0.
\end{equation}
Therefore, choosing $\beta > \max\left( \frac{18\sqrt{3}+6}{13}L,\; \max\limits_{i=1,2,\ldots,N-1}\frac{6 L^2}{\sigma_{\min}(H_i)}\right)$ and $\gamma$ as in \eqref{lemma:proximal ADMM-g-monotonicity-gamma} guarantees that $\Psi_G(x_1^{k+1}, \cdots, x_N^{k+1}, \lambda^{k+1}, x_N^k)$ monotonically decreases over $k\geq 0$.
In fact, \eqref{beta-L-decrease} can be verified as follows. By denoting $z=\beta-\frac{1}{\gamma}$, \eqref{beta-L-decrease} is equivalent to
\[
12 z^2 + 2\beta z + \left( 6L^2 + \beta L - \beta^2\right) < 0,
\]
{which holds when $\beta > \frac{18\sqrt{3}+6}{13}L$ and
$\frac{-\beta - \sqrt{13\beta^2 - 12\beta L - 72L^2}}{12} < z < \frac{-\beta + \sqrt{13\beta^2 - 12\beta L - 72L^2}}{12}$,
i.e.,
$$\frac{-13\beta - \sqrt{13\beta^2 - 12\beta L - 72L^2}}{12} < -\frac{1}{\gamma} < \frac{-13\beta + \sqrt{13\beta^2 - 12\beta L - 72L^2}}{12},$$
which holds when $\gamma$ is chosen as in \eqref{lemma:proximal ADMM-g-monotonicity-gamma}.}
\end{proof}

\begin{lemma}\label{lemma:lower-bound-I}
Suppose the sequence $\{(x_1^k,\cdots,x_N^k , \lambda^k)\}$ is generated by Algorithm \ref{alg:gadm}.
Under the same conditions as in Lemma \ref{lemma:proximal ADMM-g-monotonicity}, for any $k\geq 0$, we have
\[\Psi_G\left(x_1^{k+1}, \cdots, x_N^{k+1}, \lambda^{k+1}, x_N^k\right) \geq \sum_{i=1}^{N-1} r_i^*+f^*,\]
where $r_i^*$ and $f^*$ are defined in Assumption \ref{value-lower-bound}.
\end{lemma}
\begin{proof}
Note that from \eqref{lambda-xN-I}, we have
\begin{eqnarray*}
& & \LCal_\beta(x_1^{k+1}, \cdots, x_N^{k+1}, \lambda^{k+1}) \\
&= & \sum\limits_{i=1}^{N-1} r_i(x_i^{k+1}) + f(x_1^{k+1}, \cdots, x_N^{k+1}) - \left( \sum\limits_{i=1}^{N-1} A_i x_i^{k+1} + x_N^{k+1} - b \right)^\top\nabla_N f(x_1^{k+1}, \cdots, x_N^{k+1}) \\
& & + \frac{\beta}{2}\left\| \sum\limits_{i=1}^{N-1} A_i x_i^{k+1} + x_N^{k+1} - b \right\|^2 - \left( \sum\limits_{i=1}^{N-1} A_i x_i^{k+1} + x_N^{k+1} - b \right)^\top\left[ \left( \beta - \frac{1}{\gamma}\right) \left( x_N^k - x_N^{k+1}\right)  \right.  \\
& & \left. + \left( \nabla_N f(x_1^{k+1}, \cdots, x_{N-1}^{k+1}, x_N^k) - \nabla_N f(x_1^{k+1}, \cdots, x_N^{k+1})\right) \right] \\
&\geq & \sum\limits_{i=1}^{N-1} r_i(x_i^{k+1}) + f(x_1^{k+1}, \cdots, x_{N-1}^{k+1}, b-\sum_{i=1}^{N-1} A_i x_i^{k+1}) + \left(\frac{\beta}{2} -\frac{\beta}{6} - \frac{L}{2}\right)\left\| \sum\limits_{i=1}^{N-1} A_i x_i^{k+1} + x_N^{k+1} - b \right\|^2 \\
& & - \frac{3}{\beta}\left[ \left(\beta - \frac{1}{\gamma}\right)^2 + L^2\right] \left\| x_N^k - x_N^{k+1}\right\|^2 \\
&\geq & \sum\limits_{i=1}^{N-1} r_i^* + f^* - \frac{3}{\beta}\left[ \left(\beta - \frac{1}{\gamma}\right)^2 + L^2\right] \left\| x_N^k - x_N^{k+1}\right\|^2,
\end{eqnarray*}
where the first inequality follows from \eqref{decent-lemma}, 
and the second inequality is due to $\beta\geq 3L/2$. The desired result follows from the definition of $\Psi_G$ in \eqref{potential-function}.
\end{proof}

Now we are ready to give the iteration complexity of Algorithm \ref{alg:gadm} for finding an $\epsilon$-stationary solution of \eqref{noncvx-block-opt}.
\begin{theorem}\label{Thm:proximal ADMM-g-complexity}
Suppose the sequence $\{(x_1^k,\cdots,x_N^k , \lambda^k)\}$ is generated by Algorithm \ref{alg:gadm}.
Furthermore, suppose that $\beta$ satisfies \eqref{beta-lower-bound} and $\gamma$ satisfies \eqref{lemma:proximal ADMM-g-monotonicity-gamma}. 
Denote
\begin{eqnarray*}
&\kappa_1 := \frac{3}{\beta^2}\left[ \left(\beta - \frac{1}{\gamma}\right)^2 + L^2\right] ,\quad \kappa_2  := \left(|\beta - \frac{1}{\gamma}| + L \right)^2,\quad \kappa_3:=\max\limits_{1\leq i\leq N-1}\left(\diam(\XCal_i)\right)^2, \\
&\kappa_4 := \left( L + \beta\sqrt{N}\max\limits_{1\leq i\leq N}\left[\| A_i \|_2^2\right] + \max\limits_{1\leq i\leq N}\|H_i\|_2 \right)^2\\
\end{eqnarray*}
and
\be\label{def-tau}
\tau := \min\left\{ -\left(\frac{L+\beta}{2} - \frac{1}{\gamma} + \frac{6}{\beta}\left[ \beta - \frac{1}{\gamma}\right]^2 + \frac{3L^2}{\beta}\right), \min_{i=1,\ldots,N-1}\left\{ -\left(\frac{3L^2}{\beta} - \frac{\sigma_{\min}(H_i)}{2}\right) \right\}\right\} > 0.
\ee
Then to get an $\epsilon$-stationary solution, the number of iterations that the algorithm runs can be upper bounded by:
\be\label{def-K-proximal ADMM-g} K := \left\{ \begin{array}{ll}
\left\lceil \frac{2\max\{\kappa_1,\kappa_2,\kappa_4\cdot \kappa_3\}}{\tau\,\epsilon^2}\left(\Psi_G(x_1^1, \cdots, x_N^1, \lambda^1, x_N^0)  - \sum_{i=1}^{N-1}r_i^* - f^*\right) \right\rceil, & \mbox{for {\bf Setting 1}} \\ {\ } \\
\left\lceil \frac{2\max\{\kappa_1,\kappa_2,\kappa_4\}}{\tau\,\epsilon^2}\left(\Psi_G(x_1^1, \cdots, x_N^1, \lambda^1, x_N^0)  - \sum_{i=1}^{N-1}r_i^* - f^*\right) \right\rceil, & \mbox{for {\bf Setting 2}}
\end{array}\right. \ee
and we can further identify one iteration $\hat{k} \in \argmin\limits_{2\le k \le K+1} \sum_{i=1}^N \left(\| x_i^k - x_i^{k+1} \|^2 + \| x_i^{k-1} - x_i^k \|^2\right)$ such that
$(x_1^{\hat k}, \cdots, x_N^{\hat k} )$ is an $\epsilon$-stationary solution for optimization problem \eqref{noncvx-block-opt} {with Lagrange multiplier  $\lambda^{\hat k}$ and $A_N=I$, for Settings 1 and 2 respectively.}
\end{theorem}

\begin{proof}
For ease of presentation, denote
\be\label{theta_k}\theta_k:=\sum_{i=1}^N(\|x_i^{k} - x_i^{k+1}\|^2 + \|x_i^{k-1} - x_i^{k}\|^2).\ee
By summing \eqref{bound-potential-function-I} over $k=1,\ldots,K$, we obtain that
\be\label{diff-Psi-G}
\Psi_G(x_1^{K+1}, \cdots, x_N^{K+1}, \lambda^{K+1}, x_N^K) - \Psi_G(x_1^1, \cdots, x_N^1, \lambda^1, x_N^0) \leq -\tau\sum_{k=1}^{K} \sum_{i=1}^N \left\| x_i^k - x_i^{k+1}\right\|^2,
\ee
where $\tau$ is defined in \eqref{def-tau}.
By invoking Lemmas \ref{lemma:proximal ADMM-g-monotonicity} and \ref{lemma:lower-bound-I}, we get
\begin{eqnarray*}
\min_{2\le k \le K+1} \theta_k & \leq & \frac{1}{\tau\,K}\left[ \Psi_G(x_1^1, \cdots, x_N^1, \lambda^1, x_N^0) + \Psi_G(x_1^2, \cdots, x_N^2, \lambda^2, x_N^1)  - 2\sum_{i=1}^{N}r_i^* - 2f^*\right] \\
&\leq & \frac{2}{\tau\,K}\left[ \Psi_G(x_1^1, \cdots, x_N^1, \lambda^1, x_N^0)  - \sum_{i=1}^{N}r_i^* - f^*\right] .
\end{eqnarray*}

We now derive upper bounds on the terms in \eqref{S1-Prob-Opt-2} and \eqref{S1-Prob-Opt-3} through $\theta_k$.
Note that \eqref{lambda-xN-I} implies that
\begin{eqnarray*}
& & \| \lambda^{k+1} - \nabla_N f(x_1^{k+1}, \cdots, x_N^{k+1})\| \\
& \leq &  | \beta - \frac{1}{\gamma} | \, \|x_N^k - x_N^{k+1}\| + \|  \nabla_N f(x_1^{k+1}, \cdots, x_{N-1}^{k+1}, x_N^k) - \nabla f(x_1^{k+1}, \cdots, x_N^{k+1})  \|\\
& \le & \left[|\beta - \frac{1}{\gamma}| + L \right]\| x_N^k - x_N^{k+1}\|,
\end{eqnarray*}
which yields
\begin{equation}\label{proximal ADMM-g-bound-2}
\| \lambda^{k+1} - \nabla_N f(x_1^{k+1}, \cdots, x_N^{k+1}) \| ^2 \le \left[|\beta - \frac{1}{\gamma}| + L \right]^2 \theta_k.
\end{equation}
From Step 3 of Algorithm \ref{alg:gadm} and \eqref{lambda-norm-bound-I} it is easy to see that
\begin{eqnarray} \label{proximal ADMM-g-bound-1}
\ba{ll}
 & \left\| \sum\limits_{i=1}^{N-1} A_i x_i^{k+1} + x_N^{k+1} - b \right\|^2 = \frac{1}{\beta^2}\| \lambda^{k+1} - \lambda^k \|^2 \\
 \leq & \frac{3}{\beta^2}\left[ \beta - \frac{1}{\gamma}\right]^2\left\| x_N^k - x_N^{k+1} \right\|^2 + \frac{3}{\beta^2}\left[ \left(\beta - \frac{1}{\gamma}\right)^2 + L^2\right]\left\| x_N^{k-1} - x_N^k \right\|^2  \\ & + \frac{3L^2}{\beta^2}\sum_{i=1}^{N-1} \left\| x_i^k - x_i^{k+1}\right\|^2   \\
 \leq & \frac{3}{\beta^2}\left[ \left(\beta - \frac{1}{\gamma}\right)^2 + L^2\right]\theta_k.
\ea
\end{eqnarray}

We now derive upper bounds on the terms in \eqref{S2-Prob-Opt-1} and \eqref{S3-Prob-Opt-1} under the two settings in Table \ref{table:epsilon-stationary}, respectively.

{\bf Setting 2.} Because $r_i$ is lower semi-continuous and $\XCal_i = \RR^{n_i}$, $i=1,\ldots,N-1$, it follows from Step 1 of Algorithm \ref{alg:gadm} that there exists a general subgradient $g_i\in\partial r_i(x_i^{k+1})$ such that
\begin{eqnarray}
& &\dist \left(-\nabla_i f(x_1^{k+1},\cdots, x^{k+1}_N) + A_i^\top {\lambda}^{k+1}, \partial r_i(x_i^{k+1})\right) \label{proximal ADMM-g-bound-3} \\
& \le & \left\| g_i+\nabla_i f(x_1^{k+1},\cdots, x^{k+1}_N) - A_i^\top {\lambda}^{k+1} \right\| \nonumber \\
& = &  \bigg{\|}\nabla_i f(x_1^{k+1},\cdots, x^{k+1}_N)-\nabla_i f(x_1^{k+1},\cdots, x_i^{k+1},x_{i+1}^{k},\cdots,x^{k}_N) \nonumber \\
& &    + \, \beta A_i^{\top}\bigg{(} \sum_{j=i+1}^{N}A_j(x_j^{k+1} - x_j^{k})  \bigg{)}  - H_i (x_i^{k+1} - x_i^k)   \bigg{\|}     \nonumber \\
& \le & L \, \sqrt{\sum_{j=i+1}^N \| x_j^k - x_j^{k+1} \|^2 }  \,+ \beta\,\| A_i \|_2\, \,\sum\limits_{j=i+1}^N \| A_j \|_2 \| x_j^{k+1} - x_j^{k}\| \nonumber \\
&& + \|H_i\|_2\|x_i^{k+1} - x_i^{k}\|_2\nonumber \\
& \le &  \left( L + \beta\sqrt{N}\max_{i+1\leq j\leq N}\left[\| A_j \|_2\right]\| A_i \|_2 \right)\,\sqrt{\sum_{j=i+1}^N \| x_j^k - x_j^{k+1} \|^2 } \nonumber \\
&& +  \|H_i\|_2 \|x_i^{k+1} - x_i^{k}\|_2\nonumber \\
& \le &  \left(L +  \beta\sqrt{N}\max_{1\leq i\leq N}\left[\| A_i \|_2^2\right] + \max_{1\leq i\leq N} \|H_i\|_2 \right)\sqrt{\theta_k}.\nonumber
\end{eqnarray}
By combining \eqref{proximal ADMM-g-bound-3}, \eqref{proximal ADMM-g-bound-2} and \eqref{proximal ADMM-g-bound-1} we conclude that Algorithm \ref{alg:gadm} returns an $\epsilon$-stationary solution for \eqref{noncvx-block-opt} according to Definition \ref{r-nonconvex-unconstrained} under the conditions of Setting 2 in Table \ref{table:epsilon-stationary}.

{\bf Setting 1.} Under this setting, we know $r_i$ is Lipschitz continuous and $\XCal_i \subset \RR^{n_i}$ is convex and compact. {From Assumption \ref{assump:Lipschitz} and the fact that $\XCal_i$ is compact,} we know $r_i(x_i) + f(x_1,\cdots, x_N)$ is locally Lipschitz continuous with respect to $x_i$ for $i=1,2,\ldots, N-1$. Similar to \eqref{proximal ADMM-g-bound-3}, for any $x_i \in \XCal_i$, Step 1 of Algorithm \ref{alg:gadm} yields that
\begin{eqnarray}
& & \left(x_i-x_i^{k+1}\right)^\top\left[g_i+\nabla_i f(x_1^{k+1},\cdots, x^{k+1}_N) - A_i^\top {\lambda}^{k+1} \right] \label{proximal ADMM-g-bound-4} \\
&\geq &  \left(x_i-x_i^{k+1}\right)^\top\bigg{[}\nabla_i f(x_1^{k+1},\cdots, x^{k+1}_N)-\nabla_i f(x_1^{k+1},\cdots, x_i^{k+1},x_{i+1}^{k},\cdots,x^{k}_N) \nonumber \\
& & + \, \beta A_i^{\top}\bigg{(} \sum_{j=i+1}^{N}A_j(x_j^{k+1} -x_j^{k})  \bigg{)}  - H_i(x_i^{k+1} - x_i^k)   \bigg{]}  \nonumber \\
&\geq & -L \, \diam(\XCal_i) \sqrt{\sum_{j=i+1}^N \| x_j^k - x_j^{k+1} \|^2 }\nonumber \\
& & - \beta\| A_i \|_2\, \diam(\XCal_i)\sum\limits_{j=i+1}^N
\| A_j \|_2 \| x_j^{k+1} - x_j^{k}\|  - \diam(\XCal_i)\, \|H_i\|_2 \|x_i^{k+1} - x_i^{k}\|_2  \nonumber \\
&\geq & - \left( \beta\sqrt{N}\max_{1\leq i\leq N}\left[\| A_i \|_2^2\right] + L + \max_{1\leq i\leq N}\|H_i\|_2 \right)\max_{1\leq i\leq N-1}\left[\diam(\XCal_i)\right] \sqrt{\theta_k}, \nonumber
\end{eqnarray}
where $g_i\in\partial r_i(x_i^{k+1})$ is a general subgradient of $r_i$ at $x_i^{k+1}$. By combining \eqref{proximal ADMM-g-bound-4}, \eqref{proximal ADMM-g-bound-2} and \eqref{proximal ADMM-g-bound-1} we conclude that Algorithm \ref{alg:gadm} returns an $\epsilon$-stationary solution for \eqref{noncvx-block-opt} according to Definition \ref{r-nonconvex-constrained} under the conditions of Setting 1 in Table \ref{table:epsilon-stationary}.

\end{proof}

\begin{remark}\label{remark:gadm}
Note that the potential function $\Psi_G$ defined in \eqref{potential-function} is related to the augmented Lagrangian function.
The augmented Lagrangian function has been used as a potential function in analyzing the convergence of nonconvex splitting and ADMM methods in \cite{Ames-Hong-2016,LiPong15,HongLuoRazaviyayn15,Hong-2014-distributed,Hong16}.
See \cite{Hong16} for a more detailed discussion on this.
\end{remark}

\begin{remark}\label{remark:gadm-1}
In Step 1 of Algorithm \ref{alg:gadm}, we can also replace the function $$f(x_1^{k+1},\cdots,x_{i-1}^{k+1},x_i,x_{i+1}^k,\cdots,x_N^k)$$ by its linearization
$$f(x_1^{k+1},\cdots,x_{i-1}^{k+1},x_i^k,x_{i+1}^k,\cdots,x_N^k) + \left( x_i - x_i^k \right)^\top\nabla_i f(x_1^{k+1},\cdots, x_{i-1}^{k+1},x_i^k,x_{i+1}^k, \cdots, x_{N}^k),$$
so that the subproblem can be solved by computing the proximal mappings of $r_i$, with some properly chosen matrix $H_i$ for $i =1, \ldots, N-1$, and
the same iteration bound still holds.
\end{remark}

\subsection{Proximal majorization ADMM (proximal ADMM-m)}
Our proximal ADMM-m solves \eqref{noncvx-block-opt} under the condition that $A_N$ has full row rank. In this section, we use $\sigma_N$ to denote the smallest eigenvalue of $A_NA_N^\top$. Note that $\sigma_N >0$ because $A_N$ has full row rank. Our proximal ADMM-m can be described as follows
\begin{algorithm}[ht]
\caption{Proximal majorization ADMM (proximal ADMM-m) for solving \eqref{noncvx-block-opt} with $A_N$ being full row rank}
\label{alg:ladm}
\begin{algorithmic}[10]
\REQUIRE {Given $\left(x_1^0,x_2^0,\cdots,x_N^0\right)\in\XCal_1\times\cdots\times\XCal_{N-1}\times\RR^{n_N}$, $\lambda^0\in\RR^m$}
\FOR {$k=0,1,\ldots$}
    \STATE[Step 1] $x_i^{k+1} := \argmin_{x_i\in\XCal_i} \ \LCal_\beta(x_1^{k+1},\cdots,x_{i-1}^{k+1},x_i,x_{i+1}^k,\cdots,x_N^k,\lambda^k) + \frac{1}{2}\left\| x_i - x_i^k\right\|^2_{H_i}$ for some positive definite matrix $H_i$, $\ i=1,\ldots,N-1$
    \STATE[Step 2] $x^{k+1}_N := \argmin_{x_N} \ U(x_1^{k+1},\cdots,x_{N-1}^{k+1}, x_N,\lambda^k, x_N^k)$
    \STATE[Step 3] $\lambda^{k+1}  := \lambda^k - \beta\left( \sum_{i=1}^N A_i x_i^{k+1} - b\right)$
\ENDFOR
\end{algorithmic}
\end{algorithm}\\
In Algorithm~\ref{alg:ladm}, $U(x_1,\cdots,x_{N-1},x_N,\lambda,\bar{x})$ is defined as
\begin{eqnarray*}
U(x_1,\cdots,x_{N-1},x_N,\lambda,\bar{x}) & = & f(x_1, \cdots, x_{N-1}, \bar{x}) + \left( x_N - \bar{x} \right)^\top\nabla_N f(x_1, \cdots, x_{N-1}, \bar{x}) \\ && + \frac{L}{2}\left\| x_N - \bar{x} \right\|^2 - \left\langle\lambda, \sum_{i=1}^N A_i x_i - b\right\rangle + \frac{\beta}{2}\left\| \sum_{i=1}^N A_i x_i - b \right\|^2.
\end{eqnarray*}
Moreover, $\beta$ can be chosen  as
\be\label{beta-lower-bound-proximal ADMM-m}
\beta > \max\left\{ \frac{18L}{\sigma_N}, \; \max\limits_{1\leq i\leq N-1}\left\{  \frac{6L^2}{\sigma_N\sigma_{\min}(H_i)}\right\} \right\}.
\ee
to guarantee the convergence rate of the algorithm shown in Lemma \ref{lemma:monotonicity} and Theorem \ref{Thm:proximal ADMM-m-complexity}.


It is worth noting that the proximal ADMM-m and proximal ADMM-g differ only in Step 2: Step 2 of proximal ADMM-g takes a gradient step of the augmented Lagrangian function with respect to $x_N$, while Step 2 of proximal ADMM-m requires to minimize a quadratic function of $x_N$.

We provide some lemmas that are useful in analyzing the iteration complexity of proximal ADMM-m for solving \eqref{noncvx-block-opt}.
\begin{lemma}
Suppose the sequence $\{(x_1^k,\cdots,x_N^k,\lambda^k)\}$ is generated by Algorithm \ref{alg:ladm}. The following inequality holds
\begin{equation}\label{lambda-gap-bound-II}
\left\| \lambda^{k+1} - \lambda^k \right\|^2 \leq \frac{3L^2}{\sigma_N}\left\| x_N^k - x_N^{k+1} \right\|^2 + \frac{6L^2}{\sigma_N}\left\|  x_N^{k-1} - x_N^k \right\|^2 + \frac{3 L^2}{\sigma_N}\sum_{i=1}^{N-1}\left\| x_i^k - x_i^{k+1} \right\|^2.
\end{equation}
\end{lemma}

\begin{proof}
From the optimality conditions of Step 2 of Algorithm \ref{alg:ladm}, we have
\begin{eqnarray*}
0 & = & \nabla_N f(x_1^{k+1}, \cdots, x_{N-1}^{k+1}, x_N^k) - A_N^\top\lambda^k + \beta A_N^\top\left( \sum_{i=1}^N A_i x_i^{k+1} - b \right) - L\left( x_N^k - x_N^{k+1}\right)\nonumber \\
& = & \nabla_N f(x_1^{k+1}, \cdots, x_{N-1}^{k+1}, x_N^k) - A_N^\top\lambda^{k+1} - L\left( x_N^k - x_N^{k+1}\right), \label{proximal ADMM-m-Lambda-Formula}
\end{eqnarray*}
where the second equality is due to Step 3 of Algorithm \ref{alg:ladm}. Therefore, we have
\begin{eqnarray*}
   &  & \| \lambda^{k+1} - \lambda^k \|^2 \leq  \sigma_N^{-1}\|A_N^\top\lambda^{k+1} - A_N^\top\lambda^k \|^2\\
&\leq & \sigma_N^{-1}\| (\nabla_N f(x_1^{k+1}, \cdots, x_{N-1}^{k+1}, x_N^k) - \nabla_N f(x_1^k, \cdots, x_{N-1}^k, x_N^{k-1}) ) - L(x_N^k - x_N^{k+1}) + L( x_N^{k-1} - x_N^k)\|^2 \\
&\leq & \frac{3}{\sigma_N}\| \nabla_N f(x_1^{k+1}, \cdots, x_{N-1}^{k+1}, x_N^k) - \nabla_N f(x_1^k, \cdots, x_{N-1}^k, x_N^{k-1}) \|^2 + \frac{3L^2}{\sigma_N} (\| x_N^k - x_N^{k+1} \|^2 + \| x_N^{k-1} - x_N^k\|^2) \\
&\leq & \frac{3L^2}{\sigma_N}\| x_N^k - x_N^{k+1} \|^2 + \frac{6L^2}{\sigma_N}\|  x_N^{k-1} - x_N^k \|^2 + \frac{3 L^2}{\sigma_N}\sum_{i=1}^{N-1}\| x_i^k - x_i^{k+1} \|^2.
\end{eqnarray*}
\end{proof}

We define the following function that will be used in the analysis of proximal ADMM-m:
\[
\Psi_L\left(x_1, \cdots, x_N, \lambda, \bar{x}\right) = \LCal_\beta(x_1, \cdots, x_N, \lambda) + \frac{6L^2}{\beta\sigma_N}\left\| x_N - \bar{x}\right\|^2.
\]
Similar to the function used in proximal ADMM-g, we can prove the monotonicity and boundedness of function $\Psi_L$.

\begin{lemma}\label{lemma:monotonicity}
Suppose the sequence $\{(x_1^k,\cdots,x_N^k,\lambda^k)\}$ is generated by Algorithm \ref{alg:ladm}, where$\beta$ is chosen according to \eqref{beta-lower-bound-proximal ADMM-m}.
Then $\Psi_L(x^{k+1},\cdots,x_N^{k+1},\lambda^{k+1},x_N^k)$ monotonically decreases over $k>0$.
\end{lemma}

\begin{proof}
By Step 1 of Algorithm \ref{alg:ladm} one observes that
\be\label{ladm-monotone-proof-1}
\LCal_\beta\left(x_1^{k+1}, \cdots, x_{N-1}^{k+1}, x_N^k,\lambda^k\right) \leq \LCal_\beta\left(x_1^k, \cdots, x_N^k,\lambda^k\right) - \sum_{i=1}^{N-1} \frac{1}{2}\left\| x_i^k - x_i^{k+1}\right\|^2_{H_i},
\ee
while by Step 2 of Algorithm \ref{alg:ladm} we have
\begin{eqnarray}\label{ladm-monotone-proof-2}
\ba{lll}
0 & = & \left(x_N^k - x_N^{k+1}\right)^\top\left[ \nabla_N f(x_1^{k+1}, \cdots, x_{N-1}^{k+1}, x_N^k) - {A_N}^\top\lambda^k \right. \\
&& \left. + \beta {A_N}^\top\left( \sum_{i=1}^N A_i x_i^{k+1} - b \right) - L\left( x_N^k - x_N^{k+1}\right) \right] \\
& \leq & f(x_1^{k+1},\cdots, x_{N-1}^{k+1}, x_N^k) - f(x_1^{k+1},\cdots, x_N^{k+1}) - \frac{L}{2}\left\| x_N^k - x_N^{k+1} \right\|^2 \\
 && - \left( \sum_{i=1}^{N-1} A_i x_i^{k+1} + A_N x_N^k -b\right)^\top\lambda^k + \left( \sum_{i=1}^N A_i x_i^{k+1} - b\right)^\top\lambda^k  \\
& &  + \frac{\beta}{2}\left\| \sum_{i=1}^{N-1} A_i x_i^{k+1} + A_N x_N^k -b\right\|^2 - \frac{\beta}{2}\left\| \sum_{i=1}^N A_i x_i^{k+1} - b\right\|^2 \\
& & - \frac{\beta}{2}\left\| A_N x_N^k - A_N x_N^{k+1}\right\|^2 \\
& \leq & \LCal_{\beta}( x_1^{k+1}, \cdots, x_{N-1}^{k+1}, x_N^k, \lambda^k) - \LCal_{\beta}(x_1^{k+1}, \cdots, x_N^{k+1}, \lambda^k) - \frac{L}{2}\left\| x_N^k - x_N^{k+1}\right\|^2,
\ea
\end{eqnarray}
where the first inequality is due to \eqref{decent-lemma} and \eqref{triangle-identity}.
Moreover, from \eqref{lambda-gap-bound-II} we have
\begin{eqnarray}\label{ladm-monotone-proof-3}
& & \LCal_\beta(x_1^{k+1}, \cdots, x_N^{k+1}, \lambda^{k+1}) - \LCal_\beta( x_1^{k+1}, \cdots, x_N^{k+1}, \lambda^k) = \frac{1}{\beta}\|\lambda^k - \lambda^{k+1}\|^2  \\
& \leq & \frac{3L^2}{\beta\sigma_N}\| x_N^k - x_N^{k+1}\|^2 + \frac{6L^2}{\beta\sigma_N}\| x_N^{k-1} - x_N^k \|^2 + \frac{3L^2}{\beta\sigma_N}\sum_{i=1}^{N-1}\| x_i^k - x_i^{k+1} \|^2.\nonumber
\end{eqnarray}
Combining \eqref{ladm-monotone-proof-1}, \eqref{ladm-monotone-proof-2} and \eqref{ladm-monotone-proof-3} yields that
\begin{eqnarray*}
& & \LCal_\beta(x_1^{k+1}, \cdots, x_N^{k+1}, \lambda^{k+1}) - \LCal_\beta(x_1^k, \cdots, x_N^k, \lambda^k) \\
& \leq & \left(\frac{3L^2}{\beta\sigma_N} -\frac{L}{2} \right)\| x_N^k - x_N^{k+1}\|^2 + \sum_{i=1}^{N-1}\|x_i^k - x_i^{k+1}\|^2_{ \frac{3L^2}{\beta\sigma_N}I - \frac{1}{2}H_i } + \frac{6L^2}{\beta\sigma_N}\|x_N^{k-1} - x_N^k\|^2,
\end{eqnarray*}
which further implies that
\begin{eqnarray}\label{proximal ADMM-m-bound-potential-function}
& & \Psi_L(x_1^{k+1}, \cdots, x_N^{k+1}, \lambda^{k+1}, x_N^k) - \Psi_L(x_1^k, \cdots, x_N^k, \lambda^k, x_N^{k-1}) \\
& \leq & \left( \frac{9L^2}{\beta\sigma_N} - \frac{L}{2}\right)\left\| x_N^k - x_N^{k+1}\right\|^2 + \sum_{i=1}^{N-1}\left( \frac{3L^2}{\beta\sigma_N} - \frac{\sigma_{\min}(H_i)}{2}\right)\left\| x_i^k - x_i^{k+1} \right\|^2 < 0,\nonumber
\end{eqnarray}
where the second inequality is due to \eqref{beta-lower-bound-proximal ADMM-m}. This completes the proof.
\end{proof}


The following lemma shows that the function $\Psi_L$ is lower bounded.
\begin{lemma}\label{proximal ADMM-m-lemma:lower-bound}
Suppose the sequence $\{(x_1^k,\cdots,x_N^k,\lambda^k)\}$ is generated by Algorithm \ref{alg:ladm}. Under the same conditions as in Lemma \ref{lemma:monotonicity}, the sequence $\{\Psi_L(x^{k+1},\cdots,x_N^{k+1}, \lambda^{k+1}, x_N^k)\}$ is bounded from below.
\end{lemma}

\begin{proof}
From Step 3 of Algorithm \ref{alg:ladm} we have
\be \label{coercive-bound}
\ba{ll}
 & \Psi_L(x_1^{k+1}, \cdots, x_N^{k+1}, \lambda^{k+1}, x_N^k) \geq \LCal_\beta(x_1^{k+1}, \cdots, x_N^{k+1}, \lambda^{k+1})  \\
 = &  \sum_{i=1}^{N-1} r_i(x_i^{k+1}) + f(x_1^{k+1}, \cdots, x_N^{k+1}) - \left( \sum_{i=1}^N A_i x_i^{k+1} - b \right)^\top\lambda^{k+1} + \frac{\beta}{2}\left\| \sum_{i=1}^N A_i x_i^{k+1} - b\right\|^2  \\
 = & \sum_{i=1}^{N-1} r_i(x_i^{k+1}) + f(x_1^{k+1}, \cdots, x_N^{k+1}) - \frac{1}{\beta}(\lambda^k - \lambda^{k+1})^\top\lambda^{k+1} + \frac{1}{2\beta}\|\lambda^k - \lambda^{k+1} \|^2   \\
 = & \sum_{i=1}^{N-1} r_i(x_i^{k+1}) + f(x_1^{k+1}, \cdots, x_N^{k+1}) - \frac{1}{2\beta}\|\lambda^k\|^2 + \frac{1}{2\beta}\|\lambda^{k+1}\|^2 + \frac{1}{\beta}\|\lambda^k - \lambda^{k+1}\|^2   \\
 \geq & \sum_{i=1}^{N-1} r_i^* + f^* - \frac{1}{2\beta}\|\lambda^k\|^2 + \frac{1}{2\beta}\|\lambda^{k+1}\|^2 ,
\ea
\ee
where the third equality follows from \eqref{triangle-identity}.
Summing this inequality over $k=0,1,\ldots,K-1$ for any integer $K\geq 1$ yields that
\[
\frac{1}{K}\sum_{k=0}^{K-1} \Psi_L\left(x_1^{k+1}, \cdots, x_N^{k+1}, \lambda^{k+1}, x_N^k\right) \geq \sum_{i=1}^{N-1} r_i^* + f^* - \frac{1}{2\beta}\left\|\lambda^0\right\|^2.
\]
Lemma~\ref{lemma:monotonicity} stipulates that
$\{ \Psi_L(x_1^{k+1}, \cdots, x_N^{k+1}, \lambda^{k+1}, x_N^k) \}$ is a monotonically decreasing sequence;
the above inequality thus further implies that the entire sequence is bounded from below.
\end{proof}

We are now ready to give the iteration complexity of proximal ADMM-m, whose proof is similar to that of Theorem \ref{Thm:proximal ADMM-g-complexity}.
\begin{theorem}\label{Thm:proximal ADMM-m-complexity}
Suppose the sequence $\{(x_1^k,\cdots,x_N^k,\lambda^k) \}$ is generated by proximal ADMM-m (Algorithm \ref{alg:ladm}), and $\beta$ satisfies \eqref{beta-lower-bound-proximal ADMM-m}.
Denote
\begin{eqnarray*}
&\kappa_1 := \frac{6L^2}{\beta^2\sigma_N}, \quad \kappa_2 := 4L^2, \quad \kappa_3 := \max\limits_{1\leq i\leq N-1}\left(\diam(\XCal_i)\right)^2, \\
&\kappa_4 := \left( L + \beta\sqrt{N}\max\limits_{1\leq i\leq N}\left[\| A_i \|_2^2\right] + \max\limits_{1\leq i\leq N}\|H_i\|_2 \right)^2 ,
\end{eqnarray*}
and
\be\label{def-tau-proximal ADMM-m}
\tau := \min\left\{ -\left( \frac{9 L^2}{\beta\sigma_N}-\frac{L}{2}\right), \min_{i=1,\ldots,N-1}\left\{ -\left(\frac{3L^2}{\beta\sigma_N} - \frac{\sigma_{\min}(H_i)}{2}\right) \right\}\right\} > 0.
\ee
Then to get an $\epsilon$-stationary solution, the number of iterations that the algorithm runs can be upper bounded by:
\be\label{def-K-proximal ADMM-m} K := \left\{ \begin{array}{ll}
\left\lceil \frac{2\max\{\kappa_1,\kappa_2,\kappa_4\cdot \kappa_3\}}{\tau\,\epsilon^2}(\Psi_L(x_1^1, \cdots, x_N^1, \lambda^1, x_N^0) - \sum_{i=1}^{N-1}r_i^* - f^*) \right\rceil, & \mbox{for {\bf Setting 1}}\\ {\ } \\
\left\lceil \frac{2\max\{\kappa_1,\kappa_2,\kappa_4\}}{\tau\,\epsilon^2}(\Psi_L(x_1^1, \cdots, x_N^1, \lambda^1, x_N^0)  - \sum_{i=1}^{N-1}r_i^* - f^*) \right\rceil, & \mbox{for {\bf Setting 2}}
\end{array}\right.
\ee
and we can further identify one iteration $\hat{k} \in \argmin\limits_{2\le k \le K+1} \sum_{i=1}^N \left(\| x_i^k - x_i^{k+1} \|^2 + \| x_i^{k-1} - x_i^k \|^2\right)$,
such that $(x_1^{\hat k}, \cdots, x_N^{\hat k})$ is an $\epsilon$-stationary solution for \eqref{noncvx-block-opt} with { Lagrange multiplier $\lambda^{\hat k}$ and} $A_N$ being full row rank, {for Settings 1 and 2 respectively.}
\end{theorem}

\begin{proof}
By summing \eqref{proximal ADMM-m-bound-potential-function} over $k=1,\ldots,K$, we obtain that
\begin{equation}\label{diff-Psi-L}
\Psi_L(x_1^{K+1}, \cdots, x_N^{K+1}, \lambda^{K+1}, x_N^K) - \Psi_L(x_1^1, \cdots, x_N^1, \lambda^1, x_N^0) \leq -\tau\sum_{k=1}^{K} \sum_{i=1}^N \left\| x_i^k - x_i^{k+1}\right\|^2,
\end{equation}
where $\tau$ is defined in \eqref{def-tau-proximal ADMM-m}. From Lemma \ref{proximal ADMM-m-lemma:lower-bound} we know that there exists a constant $\Psi_L^*$ such that
$\Psi(x_1^{k+1}, \cdots, x_N^{k+1}, \lambda^{k+1}, x_N^k) \geq \Psi_L^*$ holds for any $k\geq 1$.
Therefore,
\begin{eqnarray}\label{S1-bound-Yk-Yk+1}
\min_{2\le k \le K+1} \theta_k \le
 \frac{2}{\tau\,K}\left[ \Psi_L(x_1^1, \cdots, x_N^1, \lambda^1, x_N^0) - \Psi_L^*\right],
\end{eqnarray}
where $\theta_k$ is defined in \eqref{theta_k}, i.e., for $K$ defined as in \eqref{def-K-proximal ADMM-m}, $\theta_{\hat{k}}=O(\epsilon^2)$.

We now give upper bounds to the terms in \eqref{S1-Prob-Opt-2} and \eqref{S1-Prob-Opt-3} through $\theta_k$. Note that Step 2 of Algorithm \ref{alg:ladm} implies that
\begin{eqnarray*}
& & \| A_N^{\top}\lambda^{k+1} - \nabla_N f(x_1^{k+1}, \cdots, x_N^{k+1})\| \\
& \leq &  L \, \|x_N^k - x_N^{k+1}\| + \|  \nabla_N f(x_1^{k+1}, \cdots, x_{N-1}^{k+1}, x_N^k) - \nabla_N f(x_1^{k+1}, \cdots, x_N^{k+1})  \|\\
& \le & 2L\,\| x_N^k - x_N^{k+1}\|,
\end{eqnarray*}
which implies that
\begin{equation}\label{proximal ADMM-m-bound-2}
\| A_N^\top\lambda^{k+1} - \nabla_N f(x_1^{k+1}, \cdots, x_N^{k+1}) \| ^2 \le 4L^2 \theta_k.
\end{equation}

By Step 3 of Algorithm \ref{alg:ladm} and \eqref{lambda-gap-bound-II} we have
\begin{eqnarray}
& & \left\|\sum\limits_{i=1}^{N} A_i x_i^{k+1} - b \right\|^2 = \frac{1}{\beta^2}\| \lambda^{k+1} - \lambda^k \|^2 \label{proximal ADMM-m-bound-1} \\
& \leq & \frac{3L^2}{\beta^2\sigma_N}\left\| x_N^k - x_N^{k+1} \right\|^2 + \frac{6L^2}{\beta^2\sigma_N}\left\|  x_N^{k-1} - x_N^k \right\|^2 + \frac{3 L^2}{\beta^2\sigma_N}\sum_{i=1}^{N-1}\left\| x_i^k - x_i^{k+1} \right\|^2 \nonumber \\
&\leq& \frac{6L^2}{\beta^2\sigma_N}\theta_k.\nonumber
\end{eqnarray}


The remaining proof is to give upper bounds to the terms in \eqref{S2-Prob-Opt-1} and \eqref{S3-Prob-Opt-1}. Since the proof steps are almost the same as Theorem \ref{Thm:proximal ADMM-g-complexity}, we shall only provide the key inequalities below.

{\bf Setting 2.} Under conditions in Setting 2 in Table \ref{table:epsilon-stationary}, the inequality \eqref{proximal ADMM-g-bound-3} becomes
\begin{eqnarray}
& &\dist \left(-\nabla_i f(x_1^{k+1},\cdots, x^{k+1}_N) + A_i^\top {\lambda}^{k+1}, \partial r_i(x_i^{k+1})\right) \nonumber \\
& \le &  \left(L +  \beta\sqrt{N}\max_{1\leq i\leq N}\left[\| A_i \|_2^2\right] + \max_{1\leq i\leq N}\|H_i\|_2\right)\sqrt{\theta_k}.\label{proximal ADMM-m-bound-3}
\end{eqnarray}
By combining \eqref{proximal ADMM-m-bound-3}, \eqref{proximal ADMM-m-bound-2} and \eqref{proximal ADMM-m-bound-1} we conclude that Algorithm \ref{alg:ladm} returns an $\epsilon$-stationary solution for \eqref{noncvx-block-opt} according to Definition \ref{r-nonconvex-unconstrained} under the conditions of Setting 2 in Table \ref{table:epsilon-stationary}.

{\bf Setting 1.} Under conditions in Setting 1 in Table \ref{table:epsilon-stationary}, the inequality \eqref{proximal ADMM-g-bound-4} becomes
\begin{eqnarray}\label{proximal ADMM-m-bound-4}
&  & \left(x_i-x_i^{k+1}\right)^\top\left[g_i+\nabla_i f(x_1^{k+1},\cdots, x^{k+1}_N) - A_i^\top {\lambda}^{k+1} \right]  \\
& \ge & - \left( \beta\sqrt{N}\max_{1\leq i\leq N}\left[\| A_i \|_2^2\right] + L + \max_{1\leq i\leq N}\|H_i\|_2\right)\max\limits_{1\leq i\leq N-1}\left[\diam(\XCal_i)\right] \sqrt{\theta_k}.\nonumber
\end{eqnarray}
By combining \eqref{proximal ADMM-m-bound-4}, \eqref{proximal ADMM-m-bound-2} and \eqref{proximal ADMM-m-bound-1} we conclude that Algorithm \ref{alg:ladm} returns an $\epsilon$-stationary solution for \eqref{noncvx-block-opt} according to Definition \ref{r-nonconvex-constrained} under the conditions of Setting 1 in Table \ref{table:epsilon-stationary}.

\end{proof}

\begin{remark}
In Step 1 of Algorithm \ref{alg:ladm}, we can replace the function\\
$f(x_1^{k+1},\cdots,x_{i-1}^{k+1},x_i,x_{i+1}^k,\cdots,x_N^k)$
by its linearization
$$f(x_1^{k+1},\cdots,x_{i-1}^{k+1},x_i^k,x_{i+1}^k,\cdots,x_N^k) + \left( x_i - x_i^k \right)^\top\nabla_i f(x_1^{k+1},\cdots, x_{i-1}^{k+1},x_i^k,x_{i+1}^k, \cdots, x_{N}^k).$$
Under the same conditions as in Remark \ref{remark:gadm-1},
the same iteration bound follows by slightly modifying the analysis above.
\end{remark}

\section{Extensions}\label{sec:extension}

\subsection{Relaxing the assumption on the last block variable $x_N$}
It is noted that in \eqref{noncvx-block-opt}, we have some restrictions on the last block variable $x_N$, i.e., $r_N\equiv 0$ and $A_N=I$ or is full row rank. In this subsection, we show how to remove these restrictions and consider the more general problem
\be\label{noncvx-block-opt-extension} \ba{ll}
\min &  f(x_1,x_2,\cdots, x_N) + \sum\limits_{i=1}^{N} r_i(x_i) \\
\st  & \sum_{i=1}^N A_i x_i = b, \ i=1,\ldots,N, \ea
\ee
where $x_i\in\RR^{n_i}$ and $A_i\in\RR^{m\times n_i}$, $i=1,\ldots,N$.

{
Before proceeding, we make the following assumption on \eqref{noncvx-block-opt-extension}.
\begin{assumption} \label{extension:assumption}
Denote $n = n_1+\cdots+n_N$. For any compact set $S\subseteq \RR^n$, and any sequence $\lambda^j\in \RR^m$ with $\|\lambda^j\|\rightarrow \infty$, $j=1,2,\ldots$, the following limit
\[
\lim_{j \rightarrow \infty} \dist( -\nabla f(x_1,\cdots,x_N) + A^\top \lambda^j, \sum_{i=1}^N\partial r_i(x_i)) \rightarrow \infty
\]
holds uniformly for all $(x_1,\cdots,x_N)\in S$, where $A = [A_1,\ldots,A_N]$.
\end{assumption}


Remark that the above implies $A$ to have full row-rank. Furthermore, if $f$ is continuously differentiable and $\partial r_i (S) := \bigcup_{x \in S} \partial r_i(x) $ is a compact set for any compact set $S$, and $A$ has full row rank, then Assumption~\ref{extension:assumption} trivially holds. On the other hand,
for popular non-convex regularization functions, such as SCAD, MCP and Capped $\ell_1$-norm, it can be shown that the corresponding set $\partial r_i (S)$ is indeed compact set for any compact set $S$, and so Assumption~\ref{extension:assumption} holds in all these cases.


We introduce the following problem that is closely related to \eqref{noncvx-block-opt-extension}:
\be\label{prob:penalty} \ba{ll}
\min & f(x_1,x_2,\cdots, x_N) + \sum\limits_{i=1}^{N} r_i(x_i)  + \frac{\mu(\epsilon)}{2} \| y \|^2 \\
\st  & \sum_{i=1}^{N}A_{i} x_{i} + y = b, \ i=1,\ldots,N, \ea
\ee
where $\epsilon>0$ is the target tolerance, and $\mu(\epsilon)$ is a function of $\epsilon$ which will be specified later. Now, proximal ADMM-m is ready to be used for solving \eqref{prob:penalty} because $A_{N+1}=I$ and $y$ is unconstrained. We have the following iteration complexity result for proximal ADMM-m to obtain an $\epsilon$-stationary solution of \eqref{noncvx-block-opt-extension}; proximal ADMM-g can be analyzed similarly.

\begin{theorem}\label{Thm:penalty}
Consider problem \eqref{noncvx-block-opt-extension} under Setting 2 in Table \ref{table:epsilon-stationary}. Suppose that Assumption~\ref{extension:assumption} holds, and the objective in \eqref{noncvx-block-opt-extension}, i.e., $f+\sum_{i=1}^Nr_i$, has a bounded level set. Furthermore, suppose that $f$ has a Lipschitz continuous gradient with Lipschitz constant $L$, and $A$ is of full row rank. Now let the sequence $\{(x_1^k,\cdots,x_{N}^k, y^k, \lambda^k)\}$ be generated by proximal ADMM-m for solving \eqref{prob:penalty} with initial iterates $y^0=\lambda^0=0$, and $(x_1^0,\cdots,x_N^0)$ such that $\sum_{i=1}^NA_ix_i^0=b$. Assume that the target tolerance $\epsilon$ satisfies
\be\label{theorem4.2-def-eps}
0<\epsilon < \min\left\{\frac{1}{L}, \frac{1}{6\bar\tau}\right\}, \mbox{ where } \bar\tau = \frac{1}{2}\min_{i=1,\ldots,N}\{\sigma_{\min}(H_i)\}.
\ee
Then in no more than $O(1/\epsilon^4)$ iterations we will reach an iterate $(x_1^{{\hat K}+1}, \cdots, x_{N}^{{\hat K}+1}, y^{{\hat K}+1})$ that is an $\epsilon$-stationary solution for \eqref{prob:penalty} with {Lagrange multiplier $\lambda^{{\hat K}+1}$}. Moreover,
$(x_1^{{\hat K}+1}, \cdots, x_{N}^{{\hat K}+1})$ is an $\epsilon$-stationary solution for
\eqref{noncvx-block-opt-extension} with {Lagrange multiplier $\lambda^{\hat K+1}$}.
\end{theorem}

\begin{proof}
Denote the penalty parameter as $\beta(\epsilon)$. The augmented Lagrangian function of \eqref{prob:penalty} is given by
\[
\begin{array}{lll}
\LCal_{\beta(\epsilon)}(x_1,\cdots,x_N,y,\lambda)
&:= & f(x_1,\cdots,x_N)+\sum_{i=1}^Nr_i(x_i)+\frac{\mu(\epsilon)}{2}\|y\|^2-\langle\lambda,\sum_{i=1}^NA_ix_i+y-b\rangle \\
&   & + \frac{\beta(\epsilon)}{2}\|\sum_{i=1}^NA_ix_i+y-b\|^2.
\end{array}
\]
Now we set
\be\label{theorem4.2-def-mu-beta}
\mu(\epsilon)=1/\epsilon, \mbox{ and } \beta(\epsilon)=3/\epsilon.
\ee
From \eqref{theorem4.2-def-eps} we have $\mu(\epsilon) > L$. This implies that the Lipschitz constant of the smooth part of the objective of \eqref{prob:penalty} is equal to $\mu(\epsilon)$. Then from the optimality conditions of Step 2 of Algorithm \ref{alg:ladm}, we have $\mu(\epsilon)y^k=\lambda^k, \forall k\geq 1$.

Similar to Lemma \ref{lemma:monotonicity}, we can prove that $\LCal_{\beta(\epsilon)}(x_1^k,\ldots,x_N^k,y^k,\lambda^k)$ monotonically decreases. Specifically, since $\mu(\epsilon)y^k=\lambda^k$, combining \eqref{ladm-monotone-proof-1}, \eqref{ladm-monotone-proof-2} and the equality in \eqref{ladm-monotone-proof-3} yields,
\begin{eqnarray}
&  & \LCal_{\beta(\epsilon)}(x_1^{k+1}, \cdots, x_N^{k+1}, y^{k+1}, \lambda^{k+1}) - \LCal_{\beta(\epsilon)}(x_1^k, \cdots, x_N^k, y^k,\lambda^k) \nonumber \\
&\leq & -\frac{1}{2}\sum_{i=1}^N\|x_i^k-x_i^{k+1}\|_{H_i}^2 - \left(\frac{\mu(\epsilon)}{2}-\frac{\mu(\epsilon)^2}{\beta(\epsilon)}\right)\|y^k-y^{k+1}\|^2 < 0, \label{theorem4.2-monotone}
\end{eqnarray}
where the last inequality is due to \eqref{theorem4.2-def-mu-beta}.

Similar to Lemma \ref{proximal ADMM-m-lemma:lower-bound}, we can prove that $\LCal_{\beta(\epsilon)}(x_1^k,\cdots,x_N^k,y^k,\lambda^k)$ is bounded from below, i.e., the exists a constant $\LCal^*=f^* + \sum_{i=1}^N r_i^*$ such that \[\LCal_{\beta(\epsilon)}(x_1^k,\cdots,x_N^k,y^k,\lambda^k) \ge \LCal^*, \quad \mbox{ for all } k.\]
Actually the following inequalities lead to the above fact:
\begin{eqnarray}
& & \LCal_{\beta(\epsilon)}(x_1^k,\cdots,x_N^k,y^k,\lambda^k) \nonumber\\
&= & f(x_1^k,\cdots,x_N^k) + \sum_{i=1}^N r_i(x_i^k) + \frac{\mu(\epsilon)}{2}\|y^k\|^2-\left\langle \lambda^k, \sum_{i=1}^NA_ix_i^k+y^k-b\right\rangle + \frac{\beta(\epsilon)}{2}\left\|\sum_{i=1}^NA_ix_i^k+y^k-b\right\|^2 \nonumber\\
&= & f(x_1^k,\cdots,x_N^k) + \sum_{i=1}^N r_i(x_i^k) + \frac{\mu(\epsilon)}{2}\|y^k\|^2-\left\langle \mu(\epsilon) y^k, \sum_{i=1}^NA_ix_i^k+y^k-b\right\rangle + \frac{\beta(\epsilon)}{2}\left\|\sum_{i=1}^NA_ix_i^k+y^k-b\right\|^2 \nonumber\\
&\geq & \LCal^* + \mu(\epsilon)\left[\frac{1}{2}\left\|\sum_{i=1}^NA_ix_i^k-b\right\|^2 + \left(\frac{\beta(\epsilon)-\mu(\epsilon)}{2\mu(\epsilon)}\right)\left\|\sum_{i=1}^NA_ix_i^k+y^k-b\right\|^2\right] \geq \LCal^*, \label{theorem4.2-long-inequality}
\end{eqnarray}
where the second equality is from $\mu(\epsilon)y^k=\lambda^k$, and the last inequality is due to \eqref{theorem4.2-def-mu-beta}.
Moreover, denote $\LCal^0\equiv\LCal_{\beta(\epsilon)}(x_1^0,\cdots,x_N^0,y^0,\lambda^0)$, which is a constant independent of $\epsilon$.

Furthermore, for any integer $K\geq 1$, summing \eqref{theorem4.2-monotone} over $k=0,\ldots,K$ yields
\begin{equation}\label{L-sum}
\LCal_{\beta(\epsilon)}(x_1^{K+1}, \cdots, x_N^{K+1}, y^{K+1}, \lambda^{K+1}) - \LCal^0 \leq - \bar{\tau}\sum_{k=0}^K\theta_k,
\end{equation}
where $\theta_k:=\sum_{i=1}^N\|x_i^k-x_i^{k+1}\|^2+ \|y^k-y^{k+1}\|^2$. Note that \eqref{L-sum} and \eqref{theorem4.2-long-inequality} imply that
\begin{equation}\label{theorem4.2-theta_k}
\min_{0\leq k\leq K}\theta_k \leq \frac{1}{\bar{\tau} K}\left(\LCal^0-\LCal^*\right).
\end{equation}
Similar to \eqref{proximal ADMM-g-bound-3}, it can be shown that for $i=1,\ldots,N$,
\begin{eqnarray}\label{theorem4.2-bound-3}
\begin{array}{ll}
&\dist \left(-\nabla_i f(x_1^{k+1},\cdots, x^{k+1}_N) + A_i^\top {\lambda}^{k+1}, \partial r_i(x_i^{k+1})\right) \\
\le &  \left(L +  \beta(\epsilon)\sqrt{N}\max_{1\leq i\leq N} \| A_i \|_2^2 + \max_{1\leq i\leq N} \|H_i\|_2 \right)\sqrt{\theta_k}.
\end{array}
\end{eqnarray}

Set $K= 1/\epsilon^4$ and denote $\hat{K} = \argmin_{0\leq k\leq K}\theta_k$. Then we know $\theta_{\hat{K}} = O(\epsilon^4)$. As a result,
\begin{equation}\label{theorem4.2-bound-1}
\left\|\sum_{i=1}^N A_ix_i^{\hat{K}+1}+y^{\hat{K}+1}-b\right\|^2 = \frac{1}{\beta(\epsilon)^2}\|\lambda^{\hat{K}+1}-\lambda^{\hat{K}}\|^2=\frac{\mu(\epsilon)^2}{\beta(\epsilon)^2}\|y^{\hat{K}+1}-y^{\hat{K}}\|^2\leq \frac{1}{9}\theta_{\hat{K}}=O(\epsilon^4).
\end{equation}
Note that \eqref{theorem4.2-long-inequality} also implies that $f(x_1^k,\cdots,x_N^k)+\sum_{i=1}^Nr_i(x_i^k)$ is upper-bounded by a constant. Thus, from the assumption that the level set of the objective is bounded, we know $(x_1^k,\cdots,x_N^k)$ is bounded. Then Assumption \ref{extension:assumption} implies that $\lambda^k$ bounded, which results in $\|y^k\| = O(\epsilon)$. Therefore, from \eqref{theorem4.2-bound-1} we have
\[\left\|\sum_{i=1}^N A_ix_i^{\hat{K}+1}-b\right\| \leq \left\|\sum_{i=1}^N A_ix_i^{\hat K+1}+y^{\hat K+1}-b\right\| + \left\|y^{\hat K+1}\right\| = O(\epsilon),\]
which combining with \eqref{theorem4.2-bound-3} yields that $(x_1^{\hat K+1}, \cdots, x_N^{\hat K+1})$ is an $\epsilon$-stationary solution for \eqref{noncvx-block-opt-extension} with {Lagrange multiplier $\lambda^{\hat K+1}$}, according to Definition \ref{r-nonconvex-unconstrained}.
\end{proof}

\begin{remark}
Without Assumption \ref{extension:assumption}, we can still provide an iteration complexity of proximal ADMM-m, but the complexity bound is worse than $O(1/\epsilon^4)$. To see this, note that because $\LCal_{\beta(\epsilon)}(x_1^k,\cdots,x_N^k,y^k,\lambda^k)$ monotonically decreases, the first inequality in \eqref{theorem4.2-long-inequality} implies that
\begin{equation}\label{theorem4.2-residue-bounded}
\mu(\epsilon)\frac{1}{2}\left\|\sum_{i=1}^NA_ix_i^k-b\right\|^2 \leq \LCal^0 - \LCal^*, \forall k.
\end{equation}
Therefore, by setting $K=1/\epsilon^6$, $\mu(\epsilon)=1/\epsilon^2$ and $\beta(\epsilon)=3/\epsilon^2$ instead of \eqref{theorem4.2-def-mu-beta}, and combining \eqref{theorem4.2-bound-3} and \eqref{theorem4.2-residue-bounded}, we conclude that $(x_1^{\hat K+1}, \cdots, x_N^{\hat K+1})$ is an $\epsilon$-stationary solution for \eqref{noncvx-block-opt-extension} with {Lagrange multiplier $\lambda^{\hat K+1}$}, according to Definition \ref{r-nonconvex-unconstrained}.
\end{remark}
}

\subsection{Proximal BCD (Block Coordinate Descent)}

In this section, we apply a proximal block coordinate descent method to solve the following variant of \eqref{noncvx-block-opt} and present its iteration complexity:
\be\label{prob:BCD} \ba{ll}
\min & F(x_1,x_2,\cdots,x_N):= f(x_1,x_2,\cdots, x_N)+\sum\limits_{i=1}^{N} r_i(x_i) \\
\st  & \ x_i\in \XCal_i, \ i=1,\ldots,N,\ea
\ee
where $f$ is differentiable, $r_i$ is nonsmooth, and $\XCal_i\subset\RR^{n_i}$ is a closed convex set for $i=1,2,\ldots,N$. Note that $f$ and $r_i$ can be nonconvex functions. Our proximal BCD method for solving \eqref{prob:BCD} is described in Algorithm \ref{alg:bcd}.

\begin{algorithm}[ht]
\caption{A proximal BCD method for solving \eqref{prob:BCD}}
\label{alg:bcd}
\begin{algorithmic}[10]
\REQUIRE {Given $\left(x_1^0,x_2^0,\cdots,x_N^0\right)\in\XCal_1\times\cdots\times\XCal_N$}
\FOR {$k=0,1,\ldots$}
    \STATE Update block $x_i$ in a cyclic order, i.e., for $i=1,\ldots,N$ ($H_i$ positive definite):
    \be\label{BCD-XN-Update}
        x_{i}^{k+1} := \argmin_{x_{i}\in\XCal_{i}} \ F(x_1^{k+1},\cdots,x_{i-1}^{k+1},x_i,x_{i+1}^{k},\cdots,x_{N}^{k}) + \frac{1}{2}\left\|x_{i} - x_{i}^k\right\|^2_{H_i}.
    \ee
\ENDFOR
\end{algorithmic}
\end{algorithm}


Similar to the settings in Table \ref{table:epsilon-stationary}, depending on the properties of $r_i$ and $\XCal_i$, the $\epsilon$-stationary solution for \eqref{prob:BCD} is as follows.
\begin{definition}
$(x_1^*,\ldots,x_N^*,\lambda^*)$ is called an $\epsilon$-stationary solution for \eqref{prob:BCD}, if
\begin{itemize}
\item[(i)] $r_i$ is Lipschitz continuous, $\XCal_i$ is convex and compact, and for any $x_i\in\XCal_i$, $i=1,\ldots,N$, it holds that ($g_i={\partial} r_i(x_i^*)$ denotes a generalized subgradient of $r_i$)
    \[\left(x_i-x_i^*\right)^\top\left[\nabla_i f(x_1^*,\cdots, x^*_N) + g_i \right] \geq - \epsilon;\]
\item[(ii)] or, if $r_i$ is lower semi-continuous, $\XCal_i = \RR^{n_i}$ for $i=1,\ldots,N$, it holds that
    \[\dist\left(-\nabla_i f(x_1^*,\cdots, x^*_N), {\partial} r_i(x_i^*)\right) \leq \epsilon.\]
\end{itemize}
\end{definition}

We now show that the iteration complexity of Algorithm \ref{alg:bcd} can be obtained from that of proximal ADMM-g.
By introducing an auxiliary variable $x_{N+1}$ and an arbitrary vector $b\in\RR^m$, problem \eqref{prob:BCD} can be equivalently rewritten as

\be\label{prob:BCD-equivalent} \ba{ll}
\min & f(x_1,x_2,\cdots, x_N) + \sum\limits_{i=1}^{N} r_i(x_i) \\
\st  & x_{N+1} = b, \ x_i\in \XCal_i, \ i=1,\ldots,N.\ea
\ee
It is easy to see that applying proximal ADMM-g to solve \eqref{prob:BCD-equivalent} (with $x_{N+1}$ being the last block variable) reduces exactly to Algorithm \ref{alg:bcd}. Hence, we have the following iteration complexity result of Algorithm \ref{alg:bcd} for obtaining an $\epsilon$-stationary solution of \eqref{prob:BCD}.


\begin{theorem}\label{Thm:BCD-complexity}
Suppose the sequence $\{(x_1^k,\cdots,x_N^k)\}$ is generated by proximal BCD (Algorithm \ref{alg:bcd}).
Denote
\[
\kappa_5 := ( L + \max\limits_{1\leq i\leq N} \|H_i\|_2  )^2,\,\,
\kappa_6 := \max\limits_{1\leq i\leq N} (\diam(\XCal_i) )^2.
\]
Letting
\[
K := \left\{ \begin{array}{ll}
\left\lceil \frac{\kappa_5\cdot\kappa_6}{\tau\,\epsilon^2}(\Psi_G (x_1^1, \cdots, x_N^1, \lambda^1, x_N^0)  - \sum_{i=1}^{N}r_i^* - f^*) \right\rceil &\mbox{for {\bf Setting 1} } \\ \\
\left\lceil \frac{ \kappa_5}{\tau\,\epsilon^2}(\Psi_G (x_1^1, \cdots, x_N^1, \lambda^1, x_N^0) - \sum_{i=1}^{N}r_i^* - f^*) \right\rceil &\mbox{for \bf{Setting 2}}\\
\end{array}\right.
\]
with $\tau$ being defined in~\eqref{def-tau}, and $\hat{K} := \min\limits_{1\le k \le K} \sum_{i=1}^N \left(\| x_i^k - x_i^{k+1} \|^2 \right)$, we have that
$(x_1^{\hat K}, \cdots, x_N^{\hat K})$ is an $\epsilon$-stationary solution for problem \eqref{prob:BCD}.
\end{theorem}
\begin{proof} Note that $A_1 = \cdots = A_N = 0$ and $A_{N + 1} = I$ in problem~\eqref{prob:BCD-equivalent}.
By applying proximal ADMM-g with
$\beta > \max\left\{ 18L, \; \max\limits_{1\leq i\leq N}\left\{  \frac{6L^2}{ \sigma_{\min}(H_i) }\right\} \right\}$, Theorem~\ref{Thm:proximal ADMM-g-complexity} holds. In particular,
\eqref{proximal ADMM-g-bound-3} and \eqref{proximal ADMM-g-bound-4} are valid in different settings with $\beta \sqrt{N} \max\limits_{i+1 \le j \le N+1} \left[\|A_j\|_2\right]\|A_i\|_2 = 0$ for $i = 1,\ldots, N$, which leads to the choices of $\kappa_5$ and $\kappa_6$ in the above. Moreover, we do not need to consider the optimality with respect to $x_{N+1}$ and the violation of the affine constraints, thus $\kappa_1$ and $\kappa_2$ in Theorem~\ref{Thm:proximal ADMM-g-complexity} are excluded in the expression of $K$, and the conclusion follows.
\end{proof}

\section{Numerical Experiments}\label{numerical}

We consider the following nonconvex and nonsmooth model of robust tensor PCA with $\ell_1$ norm regularization for third-order tensor of dimension $I_1 \times I_2 \times I_3$. Given an initial estimate $R$ of the CP-rank, we aim to solve the following problem:
\begin{equation}\label{formulation:RPCA-l1}
\begin{array}{cl}
\min_{A,B,C,\ZCal , \ECal , \BCal}& \|\ZCal-\llbracket A,B,C \rrbracket \|^{2}+\alpha\,\| \ECal\|_1 + \alpha_{\mathcal{N}}\| \BCal \|_F^2\\
\mbox{s.t.} &  \ZCal + \ECal + \BCal =  \mathcal{T},
\end{array}
\end{equation}
where $A \in \RR^{I_1\times R}$, $B \in \RR^{I_2\times R}$, $C \in \RR^{I_3\times R}$. 
The augmented Lagrangian function of \eqref{formulation:RPCA-l1} is given by
\begin{eqnarray*}
&& \mathcal{L}_\beta(A,B,C,\ZCal , \ECal , \BCal, \Lambda) \\
&=& \|\ZCal-\llbracket A,B,C \rrbracket \|^{2}+\alpha\,\| \ECal\|_1 + \alpha_{\mathcal{N}}\| \BCal \|^2 - \langle \Lambda, \ZCal + \ECal + \BCal - \mathcal{T} \rangle + \frac{\beta}{2} \| \ZCal + \ECal + \BCal - \mathcal{T}   \|^2.
\end{eqnarray*}

The following identities are useful for our presentation later:
$$
\|\ZCal-\llbracket A,B,C \rrbracket \|^{2} = \|{Z}_{(1)}-A (C \odot B)^{\top} \|^{2} = \|{Z}_{(2)}-B (C \odot A)^{\top} \|^{2} = \|{Z}_{(3)}-C (B \odot A)^{\top} \|^{2},
$$
where ${Z}_{(i)}$ stands for the mode-$i$ unfolding of tensor $\ZCal$ and $\odot$ stands for the Khatri-Rao product of matrices.

Note that there are six block variables in \eqref{formulation:RPCA-l1}, and we choose $\BCal$ as the last block variable. A typical iteration of proximal ADMM-g for solving \eqref{formulation:RPCA-l1} can be described as follows (we chose $H_i=\delta_i I$, with $\delta_i>0, i=1,\ldots,5$):
$$\left\{
\begin{array}{lll}
A^{k+1} &=& \left( ({Z})^k_{(1)} (C^k \odot B^k) + \frac{\delta_1}{2}A^{k} \right) \left(((C^k)^{\top}C^k)\circ ((B^k)^{\top}B^k) + \frac{\delta_1}{2} I_{R\times R} \right)^{-1}\\
B^{k+1} &=& \left( ({Z})^k_{(2)} (C^k \odot A^{k+1}) + \frac{\delta_2}{2}B^{k} \right) \left(((C^k)^{\top}C^k)\circ ((A^{k+1})^{\top}A^{k+1}) + \frac{\delta_2}{2} I_{R\times R} \right)^{-1}\\
C^{k+1} &=& \left( ({Z})^k_{(3)} ( B^{k+1} \odot  A^{k+1}) + \frac{\delta_3}{2}C^{k} \right) \left(((B^{k+1})^{\top}B^{k+1}) \circ ((A^{k+1})^{\top}A^{k+1})+ \frac{\delta_3}{2} I_{R\times R} \right)^{-1}\\
{E}_{(1)}^{k+1} & = & \mathcal{S}\left(\frac{\beta}{\beta + \delta_4}(T_{(1)} + \frac{1}{\beta}\Lambda^{k}_{(1)}-{B}_{(1)}^{k} -{Z}_{(1)}^{k})+ \frac{\delta_4}{\beta + \delta_4}{E}_{(1)}^{k},\frac{\alpha}{\beta + \delta_4}\right)\\
{Z}_{(1)}^{k+1} &= & \frac{1}{2+ 2 \delta_5 + \beta} \left( 2 A^{k+1}(C^{k+1} \odot B^{k+1})^{\top} + 2\delta_5\,({Z}_{(1)})^k + \Lambda^k_{(1)} - \beta({E}_{(1)}^{k+1}+{B}_{(1)}^k - T_{(1)})  \right)\\
{B}_{(1)}^{k+1} &=& {B}_{(1)}^{k} - \gamma\left( 2 \alpha_{\mathcal{N}} {B}_{(1)}^{k}  - \Lambda^k_{(1)} + \beta({E}_{(1)}^{k+1} + {Z}_{(1)}^{k+1} + {B}_{(1)}^{k} - T_{(1)} )\right)\\
\Lambda^{k+1}_{(1)} &=& \Lambda^k_{(1)} - \beta\left({Z}_{(1)}^{k+1} + E_{(1)}^{k+1} + B_{(1)}^{k+1} - {T}_{(1)}\right)
\end{array}
\right.
$$
where $\circ$ is the matrix Hadamard product and $\mathcal{S}$ stands for the soft shrinkage operator. The updates in proximal ADMM-m are almost the same as proximal ADMM-g except ${B}_{(1)}$ is updated as
$$
{B}_{(1)}^{k+1} = \frac{1}{L + \beta }\left( (L - 2 \alpha_{\mathcal{N}}){B}_{(1)}^{k}  + \Lambda^k_{(1)} - \beta({E}_{(1)}^{k+1} + {Z}_{(1)}^{k+1} - T_{(1)} )\right).
$$

On the other hand, note that \eqref{formulation:RPCA-l1} can be equivalently written as
\begin{equation}\label{formulation:RPCA-unconstrained}
\min_{A,B,C,\ZCal, \ECal} \ \|\ZCal-\llbracket A,B,C \rrbracket \|^{2}+\alpha\,\| \ECal\|_1 + \alpha_{\mathcal{N}}\| \ZCal + \ECal - \mathcal{T} \|_F^2,
\end{equation}
which can be solved by the classical BCD method as well as our proximal BCD (Algorithm \ref{alg:bcd}).

In the following we shall compare the numerical performance of BCD, proximal BCD, proximal ADMM-g and proximal ADMM-m for solving \eqref{formulation:RPCA-l1}.
We let $\alpha = 2/ \max\{\sqrt{I_1}, \sqrt{I_2}, \sqrt{I_3} \}$ and $\alpha_{\mathcal{N}}=1$ in model~\eqref{formulation:RPCA-l1}. We apply proximal ADMM-g and proximal ADMM-m to solve \eqref{formulation:RPCA-l1}, and apply BCD and proximal BCD to solve \eqref{formulation:RPCA-unconstrained}. In all the four algorithms we set the maximum iteration number to be $2000$, and the algorithms are terminated either when the maximum iteration number is reached or when $\theta_k$ as defined in~\eqref{theta_k} is less than $10^{-6}$. The parameters used in the two ADMM variants are specified in Table \ref{t1}.

 \begin{table*}[htb]
\centering
\begin{tabular}{c|c|c|c}
\hline
& $H_i$, $i=1,\dots, 5$& $\beta$& $\gamma$\\
\hline
proximal ADMM-g&  $\frac{1}{2} \beta \cdot I$ & $4$& $\frac{1}{\beta}$\\
\hline
proximal ADMM-m&  $\frac{2}{5} \beta \cdot I$ & $5$ & - \\
\hline
\end{tabular}
\caption{Choices of parameters in the two ADMM variants.}
\label{t1}
\end{table*}

In the experiment, we randomly generate $20$ instances for fixed tensor dimension and CP-rank.
Suppose the low-rank part $\ZCal^0$ is of rank $R_{CP}$. It is generated by
$$
\ZCal^0 = \sum \limits_{r=1}^{R_{CP}}a^{1,r}\otimes a^{2,r} \otimes a^{3,r},
$$
where vectors $a^{i,r}$ are generated from standard Gaussian distribution for $i=1,2,3$, $r = 1,\dots ,R_{CP}$.
Moreover, a sparse tensor $\ECal^0$ is generated with cardinality of $0.001\cdot I_1 I_2 I_3$ such that each nonzero component follows from standard Gaussian distribution.
Finally, we generate noise $\BCal^0 = 0.001* \hat{\BCal}$, where $\hat{\BCal}$ is a Gaussian tensor.
Then we set $\mathcal{T} = \ZCal^0 + \ECal^0 +\BCal^0$  as the observed data in~\eqref{formulation:RPCA-l1}. {A proper
initial guess $R$ of the true rank $R_{CP}$ is essential for the success of our algorithms. We can borrow the strategy in matrix completion \cite{WenYinZhang12}, and
start from a large $R$ ($R \geq R_{CP}$) and decrease it aggressively once a dramatic change in the recovered tensor $\ZCal$ is observed. We report the average performance of 20 instances of the four algorithms with initial guess $R = R_{CP}$, $R = R_{CP} + 1$ and $R = R_{CP} + \lceil 0.2*R_{CP} \rceil$ in Tables \ref{tab:rPCA1}, \ref{tab:rPCA2} and \ref{tab:rPCA3}, respectively.}

\begin{table*}[htb]\scriptsize
\centering
\begin{tabular}{c|c|c|c|c|c|c|c|c|c|c|c|c}
\hline
R$_{CP}$ &\multicolumn{3}{c|}{proximal ADMM-g}&\multicolumn{3}{|c|}{proximal ADMM-m}&\multicolumn{3}{c|}{BCD}&\multicolumn{3}{|c}{proximal BCD}\\
\hline
&Iter. & Err. & Num & Iter. & Err. & Num & Iter.  & Err. & Num & Iter.& Err. & Num \\
\hline
\multicolumn{13}{c}{Tensor Size $10\times 20 \times 30$} \\
\hline
3& 371.80 & 0.0362 & 19 & 395.25 & 0.0362 & 19 & 678.15 & 0.7093 &  1 & 292.80 & 0.0362 & 19 \\
10& 632.10 & 0.0320 & 17 & 566.15 & 0.0320 & 17 & 1292.10 & 0.9133 &  0 & 356.00 & 0.0154 & 19 \\
15& 529.25 & 0.0165 & 18 & 545.05 & 0.0165 & 18 & 1458.65 & 0.9224 &  0 & 753.75 & 0.0404 & 15\\
\hline
\multicolumn{12}{c}{Tensor Size $15\times 25 \times 40$} \\
\hline
5&516.30 & 0.0163 & 19 & 636.85 & 0.0437 & 17 & 611.25 & 0.8597 &  0 & 434.25 & 0.0358 & 18\\
10&671.80 & 0.0345 & 17 & 723.20 & 0.0385 & 17 & 1223.60 & 0.9072 &  0 & 592.60 & 0.0335 & 17 \\
20& 776.70 & 0.0341 & 16 & 922.25 & 0.0412 & 15 & 1716.05 & 0.9544 &  0 & 916.90 & 0.0416 & 14 \\
\hline
\multicolumn{12}{c}{Tensor Size $30\times 50 \times 70$} \\
\hline
8& 909.05 & 0.1021 & 13 & 1004.30 & 0.1006 & 13 & 1094.05 & 0.9271 &  0 & 798.05 & 0.1059 & 13\\
20&1304.65 & 0.1233 &  7 & 1386.75 & 0.1387 &  6 & 1635.80 & 0.9668 &  0 & 1102.85 & 0.1444 &  5 \\
40& 1261.25 & 0.0623 & 10 & 1387.40 & 0.0779 &  7 & 2000.00 & 0.9798 &  0 & 1096.80 & 0.0610 &  9\\
\hline
\end{tabular}
\caption{Numerical results for tensor robust PCA with initial guess $R = R_{CP}$}
\label{tab:rPCA1}
\end{table*}

 \begin{table*}[htb]\scriptsize
\centering
\begin{tabular}{c|c|c|c|c|c|c|c|c|c|c|c|c}
\hline
R$_{CP}$ &\multicolumn{3}{c|}{proximal ADMM-g}&\multicolumn{3}{|c|}{proximal ADMM-m}&\multicolumn{3}{c|}{BCD}&\multicolumn{3}{|c}{proximal BCD}\\
\hline
&Iter. & Err. & Num & Iter. & Err. & Num & Iter.  & Err. & Num & Iter.& Err. & Num \\
\hline
\multicolumn{13}{c}{Tensor Size $10\times 20 \times 30$} \\
\hline
3& 1830.65 & 0.0032 & 20 & 1758.90 & 0.0032 & 20 & 462.90 & 0.7763 &  0 & 1734.85 & 0.0032 & 20 \\
10& 1493.20 & 0.0029 & 20 & 1586.00 & 0.0029 & 20 & 1277.15 & 0.9133 &  0 & 1137.15 & 0.0029 & 20\\
15& 1336.65 & 0.0078 & 19 & 1486.40 & 0.0031 & 20 & 1453.30 & 0.9224 &  0 & 945.05 & 0.0106 & 19\\
\hline
\multicolumn{12}{c}{Tensor Size $15\times 25 \times 40$} \\
\hline
5&1267.10 & 0.0019 & 20 & 1291.95 & 0.0019 & 20 & 609.45 & 0.8597 &  0 & 1471.10 & 0.0019 & 20\\
10&1015.25 & 0.0019 & 20 & 1121.00 & 0.0164 & 19 & 1220.50 & 0.9072 &  0 & 1121.40 & 0.0019 & 20\\
20& 814.95 & 0.0019 & 20 & 888.40 & 0.0019 & 20 & 1716.30 & 0.9544 &  0 & 736.70 & 0.0020 & 20\\
\hline
\multicolumn{12}{c}{Tensor Size $30\times 50 \times 70$} \\
\hline
8& 719.45 & 0.0009 & 20 & 608.25 & 0.0009 & 20 & 1094.10 & 0.9271 &  0 & 508.05 & 0.0327 & 18\\
20&726.95 & 0.0088 & 19 & 817.20 & 0.0220 & 17 & 1635.10 & 0.9668 &  0 & 539.25 & 0.0254 & 17 \\
40& 1063.55 & 0.0270 & 16 & 1122.75 & 0.0322 & 15 & 1998.05 & 0.9798 &  0 & 649.10 & 0.0246 & 16\\
\hline
\end{tabular}
\caption{Numerical results for tensor robust PCA with initial guess $R = R_{CP} +1$}
\label{tab:rPCA2}
\end{table*}

 \begin{table*}[htb]\scriptsize
	\centering
	\begin{tabular}{c|c|c|c|c|c|c|c|c|c|c|c|c}
		\hline
		R$_{CP}$ &\multicolumn{3}{c|}{proximal ADMM-g}&\multicolumn{3}{|c|}{proximal ADMM-m}&\multicolumn{3}{c|}{BCD}&\multicolumn{3}{|c}{proximal BCD}\\
		\hline
		&Iter. & Err. & Num & Iter. & Err. & Num & Iter.  & Err. & Num & Iter.& Err. & Num \\
		\hline
		\multicolumn{13}{c}{Tensor Size $10\times 20 \times 30$} \\
		\hline
		3& 1740.95 & 0.0034 & 20 & 1742.35 & 0.0033 & 20 & 385.00 & 0.7320 &  0 & 1816.30 & 0.0033 & 20 \\
		10& 1932.80 & 0.0030 & 20 & 1831.20 & 0.0030 & 20 & 1324.20 & 0.9192 &  0 & 1647.80 & 0.0030 & 20 \\
		15& 1832.55 & 0.0031 & 20 & 1704.75 & 0.0031 & 20 & 1694.20 & 0.9337 &  0 & 1652.50 & 0.0031 & 20 \\
		\hline
		\multicolumn{12}{c}{Tensor Size $15\times 25 \times 40$} \\
		\hline
		5& 1249.75 & 0.0021 & 20 & 1182.65 & 0.0021 & 20 & 630.85 & 0.8594 &  0 & 1291.05 & 0.0021 & 20 \\
		10&1676.50 & 0.0021 & 20 & 1657.85 & 0.0021 & 20 & 1045.25 & 0.9150 &  0 & 1642.20 & 0.0021 & 20 \\
		20& 2000.00 & 0.0022 & 20 & 2000.00 & 0.0022 & 20 & 1891.70 & 0.9608 &  0 & 1828.20 & 0.0022 & 20 \\
		\hline
		\multicolumn{12}{c}{Tensor Size $30\times 50 \times 70$} \\
		\hline
		8& 1156.60 & 0.0009 & 20 & 908.75 & 0.0009 & 20 & 911.55 & 0.9173 &  0 & 642.25 & 0.0009 & 20 \\
		20&1156.10 & 0.0009 & 20 & 1119.50 & 0.0009 & 20 & 1588.30 & 0.9662 &  0 & 1060.25 & 0.0009 & 20 \\
		40& 1978.30 & 0.0009 & 20 & 1965.00 & 0.0009 & 20 & 1978.80 & 0.9805 &  0 & 1937.75 & 0.0009 & 20 \\
		\hline
	\end{tabular}
	\caption{Numerical results for tensor robust PCA with initial guess $R = R_{CP} + \lceil 0.2*R_{CP} \rceil$}
	\label{tab:rPCA3}
\end{table*}

In Tables \ref{tab:rPCA1}, \ref{tab:rPCA2} and \ref{tab:rPCA3}, ``Err.'' denotes the averaged relative error $\frac{\|\ZCal^*-\ZCal^0 \|_F}{\|\ZCal^0 \|_F}$ of the low-rank tensor over 20 instances, where $\ZCal^*$ is the solution returned by the corresponding algorithm; ``Iter.'' denotes the averaged number of iterations over 20 instances; ``Num'' records the number of solutions (out of 20 instances) that have relative error less than $0.01$. 

Tables \ref{tab:rPCA1}, \ref{tab:rPCA2} and \ref{tab:rPCA3} suggest that BCD mostly converges to a local solution rather than the global optimal solution, while the other three methods are much better in finding the global optimum.
{It is interesting to note that the results presented in Table~\ref{tab:rPCA3} are better than that of Table~\ref{tab:rPCA2} and Table~\ref{tab:rPCA1} when a  larger basis is allowed in tensor factorization. Moreover, in this case, the proximal BCD usually consumes less number of iterations than the two ADMM variants.}

\section*{Acknowledgements}
We would like to thank Professor Renato D. C. Monteiro and two anonymous referees for their insightful comments, which helped improve this paper significantly.

\end{document}